\documentclass[a4paper,titlepage,11pt]{article}
\usepackage{hyperref}
\usepackage[usenames,dvipsnames]{pstricks} 
\usepackage{pst-plot} 
\usepackage{amssymb,amsthm,amsmath} 
\usepackage[a4paper]{geometry}
\usepackage{datetime2}
\usepackage[utf8]{inputenc}
\usepackage[italian,english]{babel}
\usepackage{mathrsfs}
\usepackage{bbm}

\newcommand{\R}{\mathbb{R}}
\newcommand{\Z}{\mathbb{Z}}

\newcommand{\re}{\mathbb{R}}
\newcommand{\mZ}{\mathcal{Z}}
\newcommand{\N}{\mathbb{N}}
\newcommand{\ep}{\varepsilon}

\newcommand{\sPM}{\operatorname{PM}}
\newcommand{\sPMF}{\operatorname{PMF}}
\newcommand{\sDPMF}{\operatorname{DPMF}}

\newcommand{\JF}{\operatorname{\mathbb{JF}}}
\newcommand{\J}{\operatorname{\mathbb{J}}}

\newcommand{\DPMF}{\operatorname{\mathbb{DPMF}}}
\newcommand{\RDPM}{\operatorname{\mathbb{RDPM}}}
\newcommand{\RDPMF}{\operatorname{\mathbb{RDPMF}}}

\newcommand{\auto}{\mathrel{\substack{\vspace{-0.1ex}\\\displaystyle\leadsto\\[-0.9em] \displaystyle\leadsto}}}
\newcommand{\loc}{_{\mathrm{loc}}}

\newcommand{\obl}{\mathrm{Obl}}

\newcommand{\orz}{\mathrm{Hor}}

\newcommand{\weakto}{\rightharpoonup}

\newcommand{\Leb}{\ensuremath{\mathscr{L}}}

\newcommand{\mres}      {\mathop{\hbox{\vrule height 7pt width .5pt depth 0pt \vrule height .5pt width 6pt depth 0pt}}}

\newtheorem{thm}{Theorem}[section]

\newtheorem{rmk}[thm]{Remark}
\newtheorem{prop}[thm]{Proposition}
\newtheorem{defn}[thm]{Definition}

\newtheorem{lemma}[thm]{Lemma}

\title{Staircasing effect for minimizers of the one-dimensional discrete Perona-Malik functional}

\author{
Nicola Picenni\vspace{1ex}\\ 
{\normalsize Scuola Normale Superiore} \\
{\normalsize PISA (Italy)}\\
{\normalsize e-mail: \texttt{nicola.picenni@sns.it}}
}
\date{}

\begin{document}

\maketitle

\begin{abstract}

We consider the one-dimensional Perona-Malik functional, that is the energy associated to the celebrated forward-backward equation introduced by P. Perona and J. Malik in the context of image processing, with the addition of a forcing term. We discretize the functional by restricting its domain to a finite dimensional space of piecewise constant functions, and by replacing the derivative with a difference quotient.

We investigate the asymptotic behavior of minima and minimizers as the discretization scale vanishes. In particular, if the forcing term has bounded variation, we show that any sequence of minimizers converges in the sense of varifolds to the graph of the forcing term, but with tangent component which is a combination of the horizontal and vertical directions.

If the forcing term is more regular, we also prove that minimizers actually develop a microstructure that looks like a piecewise constant function at a suitable scale, which is intermediate between the macroscopic scale and the scale of the discretization.

%Our analysis relies on Gamma-convergence results for a rescaled functional, blow-up techniques, and a characterization of local minimizers for the limit problem.

\vspace{6ex}

\noindent{\bf Mathematics Subject Classification 2020 (MSC2020):} 
49J45, 35B44, 47J06, 35A35.

\vspace{6ex}

%49J45 Methods involving semicontinuity and convergence; relaxation
%35B27 Homogenization in context of PDEs; PDEs in media with periodic structure
%35A15 Variational methods applied to PDEs
%35B25 Singular perturbations in context of PDEs
%35J20 Variational methods for second-order elliptic equations
%35J35 Variational methods for higher-order elliptic equations
%47J30 Variational methods involving nonlinear operators
%49Q20 Variational problems in a geometric measure-theoretic setting
%34E10 Perturbations, asymptotics of solutions to ordinary differential equations
%35B44 Blow-up in context of PDEs
%47J06 Nonlinear ill-posed problems
%35A35 Theoretical approximation in context of PDEs
%41A30 Approximation by other special function classes


\end{abstract}

\section{Introduction}

We consider the one-dimensional Perona-Malik functional with fidelity term, that is the following
\begin{equation}\label{def:PMF}
\sPMF(u):= \int_0 ^1 \log\left(1+u'(x)^2\right)dx +\beta \int_0 ^1 (u(x)-f(x))^2\,dx,
\end{equation}
where $\beta>0$ is a positive constant and $f \in L^2((0,1))$ is a fixed function. We call \emph{forcing term} the function $f$, and \emph{fidelity term} the second integral in (\ref{def:PMF}), because it penalizes the distance between the function $u$ and the forcing term.

The principal part of (\ref{def:PMF}) is the functional
\begin{equation}\label{def:PM}
\sPM(u):= \int_0 ^1 \log\left(1+ u'(x)^2\right) dx,
\end{equation}
whose lagrangian $\phi(p)=\log(1+p^2)$ is not convex, and has a convex envelope which is identically 0 on $\R$. This implies that the functional (\ref{def:PM}) is not lower-semicontinuous, and its relaxation vanishes identically on every reasonable functional space. As a consequence, it is well-known that
$$\inf\{\sPMF(u):u \in C^1((0,1))\}=0 \qquad \forall f \in L^2((0,1)).$$

The formal gradient flow of (\ref{def:PM}) turns out to be the forward-backward parabolic equation
\begin{equation*}
u_t= \left(\frac{2 u_x}{1+u_x^2}\right)_{\! x}=\frac{2-2u_x^2}{(1+u_x^2)^2} u_{xx},
\end{equation*}
that is the one-dimensional version of the celebrated equation introduced by P. Perona and J. Malik in \cite{PeronaMalik}.

Many different approximations and regularizations of both the functional and the equation have been proposed in the literature in order to try to explain the so-called \emph{Perona-Malik paradox}, namely the fact that such an ill-posed problem turns out to behave nicely in numerical applications (see, for example, \cite{1992-SIAM-Lions,Kichenassamy,1996-Duke-DeGiorgi,2001-CPAM-Esedoglu,2007-Amann,2009-JDE-Guidotti,2012-ARMA-SmaTes}).

The regularization of (\ref{def:PMF}) by singular perturbation was considered in \cite{GP:fastPM-CdV}, where several properties concerning the asymptotic behavior of minimizers were proved.

Here we focus instead on the approximation obtained by discretization, namely for every positive integer number $n\geq 2$ we consider the functional
\begin{equation}\label{def:sPMF_n}
\sDPMF_n (u):=\int_0 ^{1-1/n} \log\left(1+\left(\frac{u(x+1/n)-u(x)}{1/n}\right)^2 \right)dx + \beta\int_0 ^1 (u(x)-f(x))^2\,dx,
\end{equation}
where the function $u$ is assumed to be constant in each interval of the form $[k/n,(k+1)/n)$, for $k \in \{0,\dots, n-1\}$.

Since the space of admissible functions $u$ is finite dimensional and the functional is continuous and coercive, we know that for every $n$ there exists at least a minimizer. The aim of this paper is to investigate the asymptotic behavior of these minimizers as $n$ tends to infinity, in the same way as it was done in \cite{GP:fastPM-CdV} with minimizers of the singularly perturbed functional.

Our results are the following.
\begin{itemize}

\item For every $f\in BV((0,1))$ we prove that any sequence $\{u_n\}$ of minimizers of (\ref{def:sPMF_n}) converges strictly in $BV((0,1))$ to the function $f$. We deduce that a suitable sequence of varifolds associated to $\{u_n\}$ converges to a varifold supported on the graph of $f$, but with tangent component consisting of a combination of horizontal and vertical lines (see Theorem~\ref{thm:varifold}).

\item For every $f\in H^1((0,1))$ we compute the asymptotic behavior of the minimum values of the functionals (\ref{def:sPMF_n}) (see Theorem~\ref{thm:min}).

\item For every $f \in C^1([0,1])$, every sequence of points $x_n\to x_\infty \in (0,1)$ and every sequence $\{u_n\}$ of minimizers of (\ref{def:sPMF_n}) we consider the sequences of blow-ups
\begin{equation}\label{def:blow-up}
y\mapsto \frac{u_n(x_n+\omega_n y) - f(x_n)}{\omega_n} \qquad\quad \mbox{and} \qquad\quad y\mapsto \frac{u_n(x_n+\omega_n y) - u_n(x_n)}{\omega_n}. \end{equation}

We prove that if $\omega_n=(\log n /n)^{1/3}$, then these sequences converge (up to subsequences) to suitable staircase-like functions, which we can characterize and depend on $f'(x_\infty)$ (see Theorem~\ref{thm:blow-up}). This means that minimizers develop a microstructure at the scale $\omega_n$, which is different from the discretization scale $1/n$.
\end{itemize}

These results correspond more or less to those obtained in \cite{GP:fastPM-CdV} for the second order approximation of the Perona-Malik functional. However, we point out that here the first result holds for a more general class of functions $f$, since the proof does not rely on the blow-up theorem. A detailed comparison between the results of the present paper and the results of \cite{GP:fastPM-CdV} is provided in Remark~\ref{rmk:result_comparison}.

%Due to the simpler nature of the problem obtained by discretization with respect to the second order approximation, here we are able to extend some results to the higher dimensional setting. Indeed, the first result holds without restrictions in any dimension, while the second one can be extended up to the characterization of the limit.

\paragraph{\textmd{\textit{Overview of the technique}}}

In the first result the convergence of the minimizers to the forcing term follows from a  truncation argument. In order to prove the varifold convergence, we divide the interval $(0,1)$ into six different zones, depending on the values of the (discrete) derivative, and we consider the restriction of minimizers to these zones. Then we show that the limit of each restriction produces a different component of the limit varifold (or vanishes).

The proof of the other two results, instead, follows the same strategy used in \cite{GP:fastPM-CdV}. Indeed we observe that the blow-ups defined in (\ref{def:blow-up}) minimize suitably rescaled versions of the functionals (\ref{def:sPMF_n}), which have a non-trivial $\Gamma$-limit that turns out to be finite only on piecewise constant functions. At this point our second and third results follow from a compactness result, some estimates on minimum values and a characterization of local minimizers for the limit functional.

\paragraph{\textmd{\textit{Structure of the paper}}}

The paper is organized as follows. In Section 2 we introduce some notation and we state the main results. In Section 3 we consider the required reascaling of the discrete Perona-Malik funcitonal, we compute its $\Gamma$-limit and we prove a compactness result and some additional properties of recovery sequences. In Section 4 we consider the limit functional, we provide some estimates for minimum values with and without boundary conditions, and we characterize the local minimizers. In Section 5 we prove our main results exploiting the results proved in Section 3 and Section 4.

\setcounter{equation}{0}
\section{Notation and statements}

%\subsection{Discrete functions and the discrete Perona-Malik functional}

\paragraph{\textmd{\textit{Discrete functions and the discrete Perona-Malik functional}}}
For every positive real number $\delta>0$ and every integer number $z \in \Z$ we set
$$I_{\delta,z}:=[z\delta,(z+1)\delta),$$
and for every real number $x\in \R$ we consider its upper and lower $\delta$-approximations, namely the numbers
\begin{equation}\label{def:a_delta-b_delta}
x_{\delta,*}:=\delta\lfloor x/\delta \rfloor,\qquad
x_\delta ^*:=\delta\lceil x/\delta \rceil,
\end{equation}
where for every real number $\alpha \in \R$ we denote by $\lfloor \alpha \rfloor$ the largest integer smaller than or equal to $\alpha$, and with $\lceil \alpha \rceil$ the smallest integer larger than or equal to $\alpha$. Let us set also
%For every open interval $(a,b)\subset \R$ let us set
$$\mathcal{Z}_\delta(a,b):=\left\{z \in \Z: I_{\delta,z}\subseteq [a_{\delta,*},b_\delta^*] \right\}=\left\{\lfloor a/\delta \rfloor,\dots,\lceil b/\delta \rceil -1\right\}.$$
%where for every real number $\alpha\in\R$ we denote by $\lfloor\alpha\rfloor$ the largest integer smaller than or equal to $\alpha$ and with $\lceil\alpha\rceil$ the smallest integer larger than or equal to $\alpha$.

Now we consider the following space of discrete functions
$$PC_\delta (a,b):=\left\{u :[a_{\delta,*},b_\delta ^*] \to \R : u \mbox{ is constant in } I_{\delta,z} \ \forall z \in \mathcal{Z}_\delta(a,b)\right\},$$
with the understanding that $u(b_\delta^*):=u(b_\delta^*-\delta)$.

%we denote by $\tilde{u}:\delta\mZ_\delta(\Omega)\to \R$ the function such that
%$$u(x)=\tilde{u}(z\delta) \qquad \forall x \in I_{\delta,z},$$
%and

%For a function $u \in PC_\delta((a,b))$ we denote the ``boundary values'' of $u$ by
%\begin{equation}\label{boundary_values}
%u(a):=\lim_{x \to a ^+} u(x) \qquad \mbox{and} \qquad u(b):=\lim_{x\to b^-} u(x),\end{equation}
%and we extend $u$ to the real line by setting $u(x):=u(a)$ for $x\leq a$ and $u(x):=u(b)$ for $x\geq b$.

We define the discrete derivative $D^{\delta}u$ as
$$D^{\delta}u(x):=\frac{u(x+\delta)-u(x)}{\delta}\qquad \forall x \in [a_{\delta,*},b_\delta^*-\delta].$$

For every positive real number $\beta>0$, every open interval $(a,b) \subset \R$, every function $f \in L^2((a,b))$ and every positive integer number $n \geq 2$ we consider the one-dimensional discrete Perona-Malik functional with fidelity term
\begin{equation}\label{def:DPMF_n}
\DPMF_n(\beta,f,(a,b),u):=\int_{a_{1/n,*}} ^{b_{1/n}^*-1/n} \log\left(1+D^{1/n} u (x)^2 \right)dx +\beta \int_{a} ^{b} (u(x)-f(x))^2\,dx,
\end{equation}
defined for every $u \in PC_{1/n}(a,b)$.

%We observe that the sum in the right-hand side can be equivalently rewritten in the form
%$$\int_{\R} \log\left(1+D^{1/n} u (x)^2 \right)dx,$$
%and also
%$$\int_{\delta\lfloor a/\delta \rfloor} ^{\delta(\lceil b/\delta\rceil-1)} \log\left(1+D^{1/n} u (x)^2 \right)dx,$$
%because the discrete derivative vanishes outside the interval $(\delta\lfloor a/\delta \rfloor, \delta(\lceil b/\delta\rceil-1)) $.

We now consider the minimum problem for the functional (\ref{def:DPMF_n}) on the interval $(0,1)$, namely
\begin{equation}\label{def:min}
m(n,\beta,f):=\min\left\{\DPMF_n (\beta,f,(0,1),u): u \in PC_{1/n}(0,1) \right\}.\end{equation}

We observe that a minimizer exists because the space $PC_{1/n}(0,1)$ has finite dimension and the functional is continuous and coercive with respect to $u$.

Moreover, for every $\beta>0$ and every $f\in L^2((0,1))$ it holds that
\begin{equation}\label{m_to_0}
\lim_{n\to +\infty} m(n,\beta,f)=0,\end{equation}
because of the sublinearity of the logarithm, and therefore, if $\{u_n\}$ is a sequence of minimizers for $m(n,\beta,f)$, we have that
\begin{equation}\label{un_L2_f}
u_n\to f \qquad \mbox{in }L^2((0,1)).
\end{equation}

\paragraph{\textmd{\textit{BV functions and strict convergence}}}
Here we introduce some notation for bounded variation functions of one real variable, and we recall the definition and some basic properties of the strict convergence.

We denote by $BV((a,b))$ the space of functions of bounded variation on an interval $(a,b)$. For a function $u$ in $BV((a,b))$ we denote by $Du$ its derivative, which is a signed measure, that can be decomposed into the sum of its diffuse part $\widetilde{D}u$ and its atomic part $D^J u$ (see \cite[Section~3.9]{AFP}). Using the Hahn decomposition, we can further decompose these measures into their positive and negative parts, so we can write
$$Du=D_+ u - D_- u, \qquad \widetilde{D} u=\widetilde{D}_+ u-\widetilde{D}_- u, \qquad D^J u=D^J_+ u - D^J_- u,$$
and, consequently,
$$D_+ u=\widetilde{D}_+ u + D^J _+ u,\qquad D_- u=\widetilde{D}_- u + D^J _- u.$$

As usual, we set $|Du|:=D_+ u + D_- u$. We also denote by $S_u:=\{x\in (a,b):Du(\{x\})\neq 0\}$ the jump set of $u$, that we divide into the two sets
$$S_{u}^+:=\{x\in (a,b):Du(\{x\})> 0\}
\qquad \mbox{and} \qquad
S_{u}^-:=\{x\in (a,b):Du(\{x\})< 0\}.$$

Finally, for every bounded variation function of one real variable $u$ we always consider representatives that are continuous outside the jump set, while for $x\in S_u$ we set
$$u(x^-):=\lim_{y\to x^-} u(y), \qquad u(x^+):=\lim_{y\to x^+} u(y),$$
and
$$\mathcal{J}_u(x):=\left[\liminf_{y\to x} u(y),\limsup_{y\to x}u(y)\right]=\left[\min\{u(x^-),u(x^+)\},\max\{u(x^-),u(x^+)\}\right].$$

We recall the definition of strict convergence of bounded variation functions (see \cite[Definition~3.14]{AFP}).

\begin{defn}[Strict convergence]
\begin{em}
Let $(a,b)\subseteq\re$ be an interval. A sequence of functions $\{u_{n}\}\subseteq BV((a,b))$ converges \emph{strictly} to some $u_{\infty}\in BV((a,b))$, and we write
\begin{equation}
u_{n}\auto u_{\infty}
\quad\text{in }BV((a,b)),
\nonumber
\end{equation}
if
\begin{equation}
u_{n}\to u_{\infty} \text{ in } L^{1}((a,b))
\qquad\text{and}\qquad
|Du_{n}|((a,b))\to|Du_{\infty}|((a,b)).
\nonumber
\end{equation}

A sequence of functions $\{u_{n}\}\subseteq BV\loc(\re)$ converges \emph{locally strictly} to some $u_{\infty}\in BV\loc(\re)$, and we write
\begin{equation}
u_{n}\auto u_{\infty}
\quad\text{in }BV\loc(\re),
\nonumber
\end{equation} 
if $u_{n}\auto u_{\infty}$ in $BV((a,b))$ for every interval $(a,b)\subseteq\re$ whose endpoints are not jump points of the limit $u_{\infty}$.
\end{em} 
\end{defn}

\begin{rmk}[Consequences of strict convergence]\label{rmk:strict}
\begin{em}

Let us assume that $u_{n}\auto u_{\infty}$ in $BV((a,b))$. Then the following facts hold true.
\begin{enumerate}
\renewcommand{\labelenumi}{(\arabic{enumi})}

\item  It turns out that $\{u_{n}\}$ is bounded in $L^{\infty}((a,b))$, and $u_{n}\to u_{\infty}$ in $L^{p}((a,b))$ for every $p\geq 1$ (but not necessarily for $p=+\infty$). 

\item For every $x\in [a,b]$, and every sequence $x_{n}\to x$, it turns out that
\begin{equation*}
\liminf_{y\to x}u_{\infty}(y)\leq\liminf_{n\to +\infty}u_{n}(x_{n})\leq
\limsup_{n\to +\infty}u_{n}(x_{n})\leq\limsup_{y\to x}u_{\infty}(y),
\end{equation*}
and in particular $u_{n}(x_{n})\to u_{\infty}(x)$ whenever $u_{\infty}$ is continuous in $x$, and the convergence is uniform in $(a,b)$ if the limit $u_{\infty}$ is continuous in $(a,b)$.

\item It turns out that $u_{n}\auto u_{\infty}$ in $BV((c,d))$ for every interval $(c,d)\subseteq(a,b)$ whose endpoints are not jump points of the limit $u_{\infty}.$

\item The positive and negative part of the distributional derivatives converge separately in the \emph{closed} interval (see~\cite[Proposition~3.15]{AFP}). More precisely, for every continuous test function $\phi:[a,b]\to\R$ it turns out that
\begin{equation*}
\lim_{n\to +\infty}\int_{[a,b]}\phi(x)\,dD_{+}u_{n}(x)=
\int_{[a,b]}\phi(x)\,dD_{+}u_{\infty}(x),
\end{equation*}
and similarly with $D_{-}u_{n}$ and $D_{-}u_{\infty}$.
\end{enumerate}

\end{em}
\end{rmk}

\paragraph{\textmd{\textit{Staircase-like functions}}}

Before stating our results we also need to introduce some notation and to recall some terminology that was used in \cite{GP:fastPM-CdV} to describe staircase-like functions.

\begin{defn}[Canonical staircases]\label{defn:staircase}
\begin{em}

Let $S:\re\to\re$ be the function defined by
\begin{equation*}
S(x):=2\left\lfloor\frac{x+1}{2}\right\rfloor
\qquad
\forall x\in\re.
\end{equation*}

For every pair $(H,V)$ of real numbers, with $H>0$, we call \emph{canonical $(H,V)$-staircase} the function $S_{H,V}:\re\to\re$ defined by
\begin{equation}
S_{H,V}(x):=V\cdot S(x/H)
\qquad
\forall x\in\R.
\label{defn:SC0}
\end{equation}

\end{em}
\end{defn}

Roughly speaking, the graph of $S_{H,V}$ is a staircase with steps of horizontal length $2H$ and vertical height $2V$. The origin is the midpoint of the horizontal part of one of the steps. The staircase degenerates to the null function when $V=0$, independently of the value of~$H$.

\begin{defn}[Translations of the canonical staircase]\label{defn:translations}
\begin{em}

Let $(H,V)$ be a pair of real numbers, with $H>0$, and let $S_{H,V}$ be the function defined in (\ref{defn:SC0}). Let $v:\re\to\re$ be a function.
\begin{itemize}

\item We say that $v$ is an \emph{oblique translation} of $S_{H,V}$, and we write $v\in\obl(H,V)$, if there exists a real number $\tau_{0}\in[-1,1]$ such that
\begin{equation}
v(x)=S_{H,V}(x-H\tau_{0})+V\tau_{0}
\qquad
\forall x\in\re.
\nonumber
\end{equation}

\item We say that $v$ is a \emph{graph translation of horizontal type} of $S_{H,V}$, and we write $v\in\orz(H,V)$, if there exists a real number $\tau_{0}\in[-1,1]$ such that
\begin{equation}
v(x)=S_{H,V}(x-H\tau_{0})
\qquad
\forall x\in\re.
\nonumber
\end{equation}

\end{itemize}

\end{em}
\end{defn}

\paragraph{\textmd{\textit{Main results}}}

We can now state our main results. The first one improves the convergence (\ref{un_L2_f}) of minimizers to the forcing term in the case in which $f$ has bounded variation.

\begin{thm}\label{thm:varifold}
Let $\beta>0$ be a positive real number and let $f \in BV((0,1))$ be a function. For every integer $n \geq 2$ let $u_n\in PC_{1/n}(0,1)$ be a minimizer for the problem (\ref{def:min}). Then the sequence $\{u_n\}$ converges to $f$ in the following senses.
\begin{enumerate}
\renewcommand{\labelenumi}{(\arabic{enumi})}

\item \emph{(Strict convergence).} It turns out that $u_{n}\auto f$ strictly in $BV((0,1))$.

\item \emph{(Convergence as varifolds).}  %Let us set
%\begin{equation}
%V_{0}^{+}:=\left\{x\in[0,1]:f'(x)>0\right\},
%\qquad
%V_{0}^{-}:=\left\{x\in[0,1]:f'(x)<0\right\}.
%\label{defn:V0+-}
%\end{equation}
Let $\widehat{u}_n:[0,1]\to \R$ denote the piecewise affine function such that $\widehat{u}_n(z/n)=u_n(z/n)$ for every $z\in\{0,\dots,n\}$, so $\widehat{u}_n '(x)=D^{1/n}u(x)$ for every $x\in (0,1) \setminus \{1/n,2/n,\dots,(n-1)/n\}$.

%Let us denote by $D_+f$ and $D_-f$ respectively the positive and the negative part of the measure $Df$, and with $D^J _+ f$ and $D^J _- f$ respectively the atomic parts of these measures. Let us set also $\widetilde{D}_+ f:=D_+f -D^J _+ f$ and $\widetilde{D}_- f:=D_-f -D^J _- f$.

%Finally, we set $S_+f:=\{x\in (0,1): D^J _+ f(\{x\})>0\}$, $S_-f:=\{x\in (0,1): D^J _- f(\{x\})<0\}$ and
%$$J_x:=\left[\liminf_{y\to x} f(y),\limsup_{y\to x} f(y)\right] \qquad \forall x \in Sf:= S_+ f \cup S_- f.$$

Then for every continuous test function
\begin{equation*}
\phi:[0,1]\times\R\times\left[-\frac{\pi}{2},\frac{\pi}{2}\right]\to\R
\end{equation*}
it turns out that
\begin{align}
\lim_{n\to +\infty}&\int_{0}^{1}
\phi\left(x,\widehat{u}_{n}(x),\arctan(\widehat{u}_n' (x))\strut\right)\sqrt{1+\widehat{u}_n'(x)^{2}}\,dx=
\nonumber\\[1ex]
&\int_{0}^{1}\phi(x,f(x),0)\,dx
+\int_0 ^1 \phi\left(x,f(x),\frac{\pi}{2}\right)d\widetilde{D}_+f(x)
+\int_0 ^1 \phi\left(x,f(x),-\frac{\pi}{2}\right)d\widetilde{D}_-f(x)\nonumber\\[1ex]
&+\sum_{x\in S_{f}^+} \int_{\mathcal{J}_f(x)} \phi\left(x,s,\frac{\pi}{2}\right) ds
+\sum_{x\in S_{f}^-} \int_{\mathcal{J}_f(x)} \phi\left(x,s,-\frac{\pi}{2}\right) ds.
\label{th:varifold}
\end{align}
\end{enumerate}

\end{thm}

\begin{rmk}[Varifold interpretation]
\begin{em}

Let us limit ourselves for a while to test functions such that $\phi(x,s,\pi/2)=\phi(x,s,-\pi/2)$ for all admissible values of $x$ and $s$. Let us observe that the function $p\mapsto\arctan(p)$ is a homeomorphism between the projective line and the interval $[-\pi/2,\pi/2]$ with the endpoints identified. Under these assumptions we can interpret the two sides of (\ref{th:varifold}) as the action of two suitable varifolds on the test function $\phi$.

In the left-hand side we have the varifold associated to the graph of $\widehat{u}_n$ in the canonical way, namely with ``weight'' (projection into $\re^{2}$) equal to the restriction of the one-dimensional Hausdorff measure to the graph of $\widehat{u}_{n}$, and ``tangent component'' in the direction of the derivative $\widehat{u}_{n}'$. In the right-hand side we have a varifold with
\begin{itemize}

\item  ``weight'' equal to the one-dimensional Hausdorff measure restricted to the complete graph of $f$ (namely the graph with the addition of the vertical segments $\{x\}\times \mathcal{J}_f(x)$ that join the extremities of the graph at jump points $x\in S_f$), multiplied by the density
\begin{equation}
\vartheta(x,y):=\begin{cases}
\frac{1+|f'(x)|}{\sqrt{1+|f'(x)|^{2}}},&\mbox{if }x\notin S_f \mbox{ and } y=f(x),\\
1 &\mbox{if }x \in S_f \mbox{ and } y\in \mathcal{J}_f(x),
\end{cases}
\nonumber
\end{equation}
with the understanding that
$$\frac{1+|f'(x)|}{\sqrt{1+|f'(x)|^{2}}}=1,$$
if $f'(x)=\pm\infty$. In this way, we can give a meaning to the this expression for every $x\in (0,1)$ outside a set that is negligible with respect to both the Lebesgue measure and $|Df|$. In particular, $\vartheta(x,y)$ is well defined for $\mathcal{H}^1$ almost every $(x,y)$ in the complete graph of $f$.

\item  ``tangent component'' in the point $(x,y)$ equal to
\begin{equation}
T(x,y):=\begin{cases}
\lambda(x)\,\delta_{(1,0)}+(1-\lambda(x))\,\delta_{(0,1)},&\mbox{if }x\notin Sf \mbox{ and } y=f(x),\\
\delta_{(0,1)} &\mbox{if }x \in Sf \mbox{ and } y\in \mathcal{J}_f(x).
\end{cases}
\nonumber
\end{equation}

where $\delta_{(1,0)}$ and $\delta_{(0,1)}$ are the Dirac measures concentrated in the horizontal direction $(1,0)$ and in the vertical direction $(0,1)$, respectively, and
$$\lambda(x):=\frac{1}{1+|f'(x)|},$$
with the understanding that $\lambda(x)=0$ if $f'(x)=\pm\infty$. As above, it turns out that $T(x,y)$ is well defined for $\mathcal{H}^1$ almost every $x$ in the complete graph of $f$.
\end{itemize}
 
It follows that statement~(2) of Theorem~\ref{thm:varifold} above is a reinforced version of varifold convergence. The reinforcement consists in considering the vertical tangent line in the direction $(0,1)$ as different from the vertical tangent line in the direction $(0,-1)$.

\end{em}
\end{rmk}

In order to state the next results, for every integer $n \geq 2$ let us set
\begin{equation}\label{def:omega_n-delta_n}
\omega(n):=\left(\frac{\log n}{n}\right)^{1/3} \qquad \mbox{and}\qquad \delta(n):=\frac{1}{n\omega(n)}=\frac{1}{n^{2/3}\, (\log n)^{1/3}}.
\end{equation}

The next theorem concerns the asymptotic behavior of minimum values and applies to more regular forcing terms $f \in H^1((0,1))$.

%In order to state the result, we need to introduce two infinitesimal sequences of real numbers $\{\ep_n\}$ and $\{\omega_n\}$ as follows. For every integer $n\geq 2$ sufficiently large, let $\ep_n \in (0,1/2)$ be such that
%$$\ep_n ^3 |\log(\ep_n)|^{1/2}=\frac{1}{n},$$
%and let us set
%$$\omega_n := \ep_n |\log(\ep_n)|^{1/2}.$$

%We observe that $\ep_n\to 0$ and $\omega_n\to 0$.

%The result is the following.

\begin{thm}\label{thm:min}
Let $\beta>0$ be a positive real number and let $f \in H^1((0,1))$. 
Then the minimum value defined in (\ref{def:min}) satisfies
$$\lim_{n\to +\infty} \frac{m(n,\beta,f)}{\omega(n) ^2}= \beta^{1/3} \int_0 ^1 |f'(x)|^{2/3}\,dx. $$
\end{thm}

Finally, in the last main result we consider the case in which $f$ is of class $C^1$, and we prove that the blow-ups of minimizers (namely the functions defined in (\ref{def:blow-up})) converge to translations of a suitable staircase, with parameters depending on the derivative of $f$ in the center of the blow-ups.

\begin{thm}\label{thm:blow-up}
Let $\beta>0$ be a positive real number and let $f \in C^1([0,1])$ be a function. For every integer $n\geq 2$ let $u_n \in PC_{1/n}(0,1)$ be a minimizer for the functional (\ref{def:DPMF_n}). Let also $\{x_n\}\subset [0,1]$ be a family of points such that $x_n=k_n/n$ for some $k_n\in\{0,\dots,n-1\}$ and $x_n \to x_\infty\in(0,1)$.

Let us consider the functions $v_n,w_n \in PC_{\delta(n)}(-x_n/\omega(n),(1-x_n)/\omega(n))$ defined by
\begin{eqnarray}
w_n(y)&:=&\frac{u_n(x_n+\omega(n) y)-f(x_n)}{\omega(n)},\\
v_n(y)&:=&\frac{u_n(x_n+\omega(n) y)-u_n(x_n)}{\omega(n)},
\end{eqnarray}
and let us consider the canonical $(H,V)$-staircase with parameters
\begin{equation}\label{def:H_V_x}
H:=\left(\frac{1}{\beta |f'(x_\infty)|^2}\right)^{1/3}, \qquad V:=f'(x_\infty)H,
\end{equation}
with the understanding that this staircase is identically equal to $0$ when $f'(x_\infty)=0$.

Then the following statements hold true.

\begin{enumerate}
\renewcommand{\labelenumi}{(\arabic{enumi})}
\item The sequence $\{w_n\}$ is relatively compact with respect to locally strict convergence, and every limit point is an oblique translation of the canonical $(H,V)$-staircase.

More precisely, for every sequence $\{n_k\}$ of integer numbers such that $n_k\to +\infty$, there exist a subsequence $\{n_{k_h}\}$ and a function $w_\infty \in \obl(H,V)$ such that 
$$w_{n_{k_h}}\auto w_\infty \qquad \mbox{in }BV_{loc}(\R).$$

\item The sequence $\{v_n\}$ is relatively compact with respect to locally strict convergence, and every limit point is a graph translation of horizontal type of the canonical $(H,V)$-staircase.

More precisely, for every sequence $\{n_k\}$ of integer numbers such that $n_k\to +\infty$, there exist a subsequence $\{n_{k_h}\}$ and a function $v_\infty \in \orz(H,V)$ such that 
$$v_{n_{k_h}}\auto v_\infty \qquad \mbox{in }BV_{loc}(\R).$$

\item For every $w \in \obl(H,V)$ there exists a sequence $\{x_n'\}\subset (0,1)$ such that
\begin{equation}\label{x'-x<H}
\limsup_{n\to +\infty} \frac{x_n'-x_n}{\omega(n)}\leq H,
\end{equation}
and
\begin{equation}\label{fakebu_to_w}
\frac{u_n(x_n '+\omega(n) y)-f(x_n')}{\omega(n)} \auto w(y) \quad \mbox{in } BV_{loc}(\R).
\end{equation}

\item For every $v \in \orz(H,V)$ there exists a sequence $\{x_n'\}\subset (0,1)$ such that (\ref{x'-x<H}) holds and
\begin{equation}\label{truebu_to_v}
\frac{u_n(x_n '+\omega(n) y)-u_n(x_n')}{\omega(n)} \auto v(y) \quad \mbox{in } BV_{loc}(\R).
\end{equation}
\end{enumerate}
\end{thm}

%[Piattezza dei tratti orizzontali? Forse è già fin troppo lungo]

\begin{rmk}[Interpretation of Theorem~\ref{thm:blow-up} in the framework of \cite{2001-CPAM-AM}]
\begin{em}
As it was pointed out in \cite[Remark~2.12]{GP:fastPM-CdV}, from this statement one can easily deduce some properties of the sequence of maps $U_n:(0,1)\to BV_{loc}(\R)$ that associate to any point $x\in (0,1)$ the functions
$$U_n(x)(y):=\frac{u_n(x+\omega(n) y)-f(x)}{\omega(n)},$$
where $u_n$ is extended to $\R$ by setting $u_n(x)=u_n(0)$ for $x<0$ and $u_n(x)=u_n(1)$ for $x>1$.

In particular, if we endow $BV_{loc}(\R)$ with any distance that induces the locally strict convergence, then Theorem~\ref{thm:blow-up} implies that for every interval $[a,b]\subset (0,1)$ the graphs of the restrictions of $U_n$ to $[a,b]$ converge in the Hausdorff sense to the set $\{(x,S)\in [a,b]\times BV_{loc}(\R): S\in \obl(H(x),V(x))\}$, where $H(x)$ and $V(x)$ are defined as in (\ref{def:H_V_x}) with $x_\infty=x$, and the understanding that $\obl(H(x),V(x))$ contains only the null function if $f'(x)=0$.

As it was also observed in \cite[Remark~2.12]{GP:fastPM-CdV}, from this we deduce also that the Young measures associated to the sequence $\{U_n\}$ converge to the map which associates to every $x$ the probability supported on $\obl(H(x),V(x))$ that is invariant by oblique translations. This is the notion of convergence introduced in \cite{2001-CPAM-AM} to deal with multi-scale problems, with the difference that here we consider Young measures on $BV_{loc}(\R)$ endowed with the locally strict convergence, which is much stronger than the weak* convergence in $L^\infty$.
\end{em}
\end{rmk}

\begin{rmk}[Differences and analogies with the second order approximation]\label{rmk:result_comparison}
\begin{em}
Let us compare the above results to the results obtained in \cite{GP:fastPM-CdV} for the second order approximation of the Perona-Malik functional.
\begin{itemize}

\item Theorem~\ref{thm:varifold} corresponds to \cite[Theorem~2.14]{GP:fastPM-CdV}. Here, however, we only assume that $f \in BV((0,1))$, while in \cite[Theorem~2.14]{GP:fastPM-CdV} it is assumed that $f \in C^1([0,1])$.

This is an important difference, because in \cite{GP:fastPM-CdV} the proof of the strict convergence was based on the blow-up result, which holds only for $C^1$ forcing terms, while here the proof of the strict convergence relies on a truncation argument, which does not require regularity of the forcing term, but can not work when second derivatives are involved.

In both cases, the varifold convergence is then deduced from the strict convergence, but here it requires more work, because the forcing term can have jumps.

\item Theorem~\ref{thm:min} corresponds to \cite[Theorem~2.2]{GP:fastPM-CdV}. Again, here we assume only that $f \in H^1((0,1))$, while in \cite[Theorem~2.2]{GP:fastPM-CdV} it is assumed that $f \in C^1([0,1])$.

However, in this case the proofs are similar, and actually it would be possible to extend \cite[Theorem~2.2]{GP:fastPM-CdV} to forcing terms of class $H^1$, basically with the same modifications that we have introduced here, as it was already claimed in \cite[Section~8]{GP:fastPM-CdV}.

\item Theorem~\ref{thm:blow-up} corresponds to \cite[Theorem~2.9]{GP:fastPM-CdV}. Here the only difference is that vertical translations of the canonical staircase cannot arise as limits of blow-ups with the discrete approximation, because discrete functions are not continuous, but have actual jumps near the jump points of the staircase.

As for the proof, the ideas are basically the same, with the notable difference that in the case of the discrete approximation it is easier to obtain a uniform bound on the rescaled energy, so here we do not need the iterative argument of \cite[Proposition~6.5]{GP:fastPM-CdV}.

\end{itemize}
\end{em}
\end{rmk}

\setcounter{equation}{0}
\section{The rescaled functionals}

\subsection{Gamma-convergence and compactness}

This section deals with a rescaled version of the discrete Perona-Malik functionals that arises naturally in the study of the minimum problem (\ref{def:min}). The convergence of different rescalings of the discrete Perona-Malik functional (and of the corresponding gradient-flows) has been widely investigated in the last decades (see \cite{2001-CPAM-Esedoglu,2003-M3AS-MN,GG:grad-est,2008-JDE-BNPT,2009-Calcolo-BNPT,2011-M3AS-BelNovPao,2011-SIAM-SlowSD,2018-ActApplMath-BraVal,2023-PM-TV}). The particular rescaling that we need here is the following
\begin{equation}\label{def:RDPM_n}
\RDPM_n ((a,b),u):=
\begin{cases}
\displaystyle{\frac{1}{\omega(n)^2} \int_{a_{\delta(n),*}} ^{b_{\delta(n)}^*-\delta(n)} \log\left(1+D^{\delta(n)} u(x)^2\right) \,dx} &\mbox{if }u \in PC_{\delta(n)}(a,b),\\
+\infty &\mbox{otherwise},
\end{cases}
\end{equation}
where $\omega(n)$ and $\delta(n)$ are defined by (\ref{def:omega_n-delta_n}) and $a_{\delta(n),*}$ and $b_{\delta(n)}^*$ are defined according to (\ref{def:a_delta-b_delta}).

The $\Gamma$-limit of functionals analogous to (\ref{def:RDPM_n}) has been computed in \cite{2009-Calcolo-BNPT}, relying on the characterization of the local minimizers with Dirichlet boundary conditions. Actually, these functionals also fit, at least partially, into the framework of \cite{2001-GM}, where a general theory for discrete functionals with convex-concave lagrangians was developed. More precisely, the functionals (\ref{def:RDPM_n}) satisfy the assumptions of the liminf inequality \cite[Theorem~3.1]{2001-GM}, but not the assumptions for the pointwise convergence and the limsup inequality \cite[Theorem~3.2]{2001-GM}, whose proof anyway is quite elementary.

However, here we also need a compactness statement, and some additional properties of recovery sequences that are not provided in those papers, nor can be easily deduced from the proofs contained therein.

Hence we include here a different proof of the $\Gamma$-convergence, a compactness result and some additional properties of recovery sequences for (\ref{def:RDPM_n}) which are inspired by the ideas used in \cite{1998-ABG,2008-TAMS-BF,GP:fastPM-CdV} with the second order approximation.
%[We later extend these results to higher dimensions in Section \ref{?}. Ma forse anche no...]

In order to give precise statements, we need to define the space $S((a,b))$ of step functions with finitely many jumps on an interval $(a,b)$, that is the space of all functions $u:(a,b) \to \R$ for which there exists a finite (or empty) set $S_u=\{s_1,\dots,s_m\}\subset (a,b)$, a constant $c\in\R$ and a function $J_u:S_u\to \R\setminus\{0\}$ such that
$$u(x)=c+\sum_{s \in S_u}J_u(x)\mathbbm{1}_{[s,b)}(x)\qquad \forall x \in (a,b).$$

Equivalently, the space $S((a,b))$ may be defined as the space of functions in $u \in BV((a,b))$ such that $\widetilde{D}u=0$ and $S_u$ is finite.

Consistently with the notation that we introduced for BV functions, we call \emph{jump points} of $u$ the points in $S_u$ and for every jump point $x\in S_u$ we denote by $u(x^-)$ and $u(x^+)$ the two values of $u$ near $x$.

Similarly, we write $u(a)$ and $u(b)$ to the denote the values of a function $u\in S((a,b))$ near the endpoints of the interval, even if the function is defined only in the interior.

We also define the set $S_{loc}(\R)$ of all functions $u:\R \to \R$ whose restriction to every bounded interval $(a,b)$ belongs to $S((a,b))$.

For a function $u \in S((a,b))$ and an open interval $\Omega \subseteq (a,b)$ we define the functional
\begin{equation}\label{def:J}
\J(\Omega,u):=\mathcal{H}^0(S_u\cap \Omega),
\end{equation}
that simply counts the number of jump points of $u$ in $\Omega$.

The result is the following.

\begin{thm}\label{thm:gamma-convergence}
Let $(a,b)\subset \R$ be an interval and $p \in [1,+\infty)$ be a real number. Then the following statements hold true.
\begin{enumerate}
\renewcommand{\labelenumi}{(\arabic{enumi})}

\item Let $\{n_k\}$ be a sequence of integers such that $n_k \to +\infty$, and let $\{u_k\}\subset L^p((a,b))$ be a family of functions such that
$$\sup_{k \in \N}\left\{ \RDPM_{n_k}((a,b),u_k)+\|u_k\|_\infty \right\}<+\infty.$$

Then there exist a sequence of positive integers $k_h\to +\infty$ and a function $u \in S((a,b))$ such that $u_{k_h}\to u$ in $L^p ((a,b))$.

\item Let us extend the functional $\J((a,b),u)$ to $L^p((a,b))$ by setting it equal to $+\infty$ outside its original domain. Then for every function $u\in L^p((a,b))$ and every sequence $\{u_n\}$ of functions such that $u_n \in PC_{\delta(n)}(a,b)$ for every $n\geq 2$ and $u_n \to u$ in $L^p((a,b))$ it turns out that
\begin{equation}\label{liminf_RDPM}
\liminf_{n \to +\infty} \RDPM_n((a,b),u_n)\geq \frac{4}{3} \J((a,b),u).\end{equation}

\item For every function $u\in S((a,b))$ there exists a sequence $\{u_n\}$ of functions such that $u_n \in PC_{\delta(n)}(a,b)$ for every $n\geq 2$, $u_n \to u$ in $L^p((a,b))$ and
\begin{equation}\label{limsup_RDPM}
\limsup_{n \to +\infty} \RDPM_n((a,b),u_n)\leq \frac{4}{3} \J((a,b),u).\end{equation}

\item Let $u \in S((a,b))$ be a nonconstant function and let $\{A_n\}\subset \R$ and $\{B_n\}\subset \R$ be two sequences of real numbers such that $A_n\to u(a)$ and $B_n\to u(b)$. Then there exists a sequence of functions $\{u_n\}$ such that $u_n \in PC_{\delta(n)}(a,b)$ for every $n \geq 2$ and the following hold
\begin{gather}
u_n(a)=A_n,\qquad u_n(b)=B_n,\qquad u_n \to u \quad \mbox{in }L^p((a,b)),\nonumber\\[0.5ex]
\lim_{n\to +\infty} \RDPM_n((a,b),u_n)=\frac{4}{3}\J((a,b),u).
\label{eq:u_delta_recovery}
\end{gather}

If $u \in S((a,b))$ is constant, the same holds if we replace (\ref{eq:u_delta_recovery}) with
$$\limsup_{n\to +\infty} \RDPM_n((a,b),u_n)\leq \frac{4}{3}.$$

\item Let $u \in S((a,b))$ be a function, $\{n_k\}$ be a sequence of integers such that $n_k\to +\infty$ and $\{u_k\}$ be a sequence of functions such that $u_k\in PC_{\delta(n_k)}(a,b)$ for every $k\in\N$, $u_k\to u$ in $L^p((a,b))$ and
\begin{equation}\label{hp:u_n_recovery}
\lim_{k\to +\infty}\RDPM_{n_k}((a,b),u_k))=\frac{4}{3}\J((a,b),u).
\end{equation}

Then actually $u_k\auto u$ strictly in $BV((a,b))$.

Moreover, for every compact set $K\subset \R$ such that $u^{-1}(K)=\emptyset$ it turns out that also $u_k ^{-1} (K)=\emptyset$ when $k$ is sufficiently large.

\item Let $u\in S((a,b))$ be a constant function, $\{n_k\}$ be a sequence of integers such that $n_k \to +\infty$ and $\{u_k\}$ be a sequence of functions such that $u_k\in PC_{\delta(n_k)}(a,b)$ for every $k\in\N$, $u_k\to u$ in $L^p((a,b))$ and
\begin{equation}\label{hp:u_n_recovery_weak}
\limsup_{k\to +\infty}\RDPM_{n_k}((a,b),u_k))\leq \frac{4}{3}.
\end{equation}

Then actually $u_k\auto u$ strictly in $BV((c,d))$ for every interval $(c,d)$ with $a<c<d<b$. In particular, since $u$ is continuous, the convergence is also locally uniform.

\end{enumerate}
\end{thm}

\begin{proof}
We start with the compactness statement (1), whose proof contains the main ideas also for the proof of the subsequent statements.

\paragraph{\textmd{\textit{Statement (1)}}}
For every $k \in \N$ let us set
\begin{equation*}
A_k:=\left\{z \in \mathcal{Z}_{\delta(n_k)}(a,b) : \left|D^{\delta(n_k)}u_k (z\delta(n_k))\right| > \frac{1}{\delta(n_k) (\log n_k)^4} \right\},\end{equation*}

We observe that
\begin{eqnarray}
\limsup_{k\to+\infty} \RDPM_{n_k}((a,b),u_k)
&\geq& \limsup_{k\to+\infty} \sum_{z \in A_k} \frac{1}{\omega(n_k) ^2 } \log\left(1+ D^{\delta(n_k)} u_k(z\delta(n_k) )^2 \right) \delta(n_k) \nonumber\\
&\geq&\limsup_{k\to+\infty} \sum_{z \in A_k} \frac{1}{\log n_k} \log\left(1+ \frac{n_k ^{4/3}}{ (\log n_k)^{22/3} }\right)\nonumber\\
&=& \limsup_{n\to+\infty} \frac{4}{3} \mathcal{H}^0 (A_k).
\label{est:RDPM_An}
\end{eqnarray}

Since the left-hand side is finite, we deduce that the cardinality of $A_k$ is uniformly bounded and hence there exists a subsequence (not relabelled) such that $\mathcal{H}^0(A_{k})$ is constant, namely $A_{k}$ consists of $m$ integers $z^1 _{k}<\dots<z^m _{k}$.

So let us set $x_k ^i:= \delta(n_k) (z_k ^i+1)$ and $J_k ^i := u_{k}((z^i_k +1)\delta(n_k)) - u_{k}(z^i_k \delta(n_k))$ for every $k\in \N$ and every $i \in \{1,\dots,m\}$. Up to a further subsequence, we can assume that $x_k ^i\to x^i\in [a,b]$ and $J^i _k \to J^i\in \R$ as $k\to +\infty$ for every $i$ and also $u_k(a)\to c\in \R$.

Now we consider the function $u \in S((a,b))$ defined by
\begin{equation}\label{def:u_limit}
u(x)=c+\sum_{i=1} ^m J^i \mathbbm{1}_{[x^i,b)}(x),
\end{equation}
and we claim that $u_{k}\to u$ in $L^p((a,b))$. We point out that some of the values $J^i$ could vanish, some of the points $x^i$ could coincide, while others could be located at the boundary $\{a,b\}$, so the jump set $S_u$ might consists of less than $m$ points. In any case, (\ref{def:u_limit}) is a well-defined function that belongs to $S((a,b))$.

To prove our claim, let us introduce the functions
$$v_k(x)=u_{k}(a)+\sum_{i=1} ^m  J^i_k \mathbbm{1}_{[x_k ^i,b_{\delta(n_k)}^*]}(x) \qquad \forall x \in [a_{\delta(n_k),*},b_{\delta(n_k)}^*].$$

We observe that $v_k \in PC_{\delta(n_k)}(a,b)$ and that $v_k\to u$ in $L^p((a,b))$ because each addendum in the sum converges to the corresponding addendum in (\ref{def:u_limit}).

Therefore it is enough to prove that $\|u_{k}-v_h\|_{L^p((a,b))} \to 0$.

To this end, we introduce the sets
\begin{eqnarray*}
B_k&:=&\left\{z \in \mathcal{Z}_{\delta(n_k)}(a,b) : \frac{1}{\log n_k} \leq \left|D^{\delta(n_k)} u_{k}(z\delta(n_k))\right| \leq \frac{1}{\delta(n_k) (\log n_k)^4} \right\},\\
C_k&:=&\left\{z \in \mathcal{Z}_{\delta(n_k)}(a,b) : \left| D^{\delta(n_k)} u_{k}(z\delta(n_k))\right| < \frac{1}{\log n_k} \right\}.
\end{eqnarray*}

We can estimate the cardinality of $B_k$ in the following way
\begin{eqnarray*}
\RDPM_{n_k}((a,b),u_k)&\geq&\sum_{z \in B_k} \frac{1}{\omega(n_k) ^2 } \log\left(1+ D^{\delta(n_k)} u_k(z\delta(n_k))^2 \right) \delta(n_k) \\
&\geq&\sum_{z \in B_k} \frac{1}{\log n_k} \log\left(1+ \frac{1}{(\log n_k )^2 }\right)\\
&=&\frac{1}{\log n_k} \log\left(1+ \frac{1}{(\log n_k )^2 }\right) \mathcal{H}^0 (B_k).
\end{eqnarray*}

As a consequence we obtain that
\begin{eqnarray}
\sum_{z \in B_k} \left|u_{k}((z+1)\delta(n_k))- u_{k}(z\delta(n_k))\right|
&=& \sum_{z \in B_k} \left| D^{\delta(n_k)} u_k (z\delta(n_k))\right| \delta(n_k) \nonumber\\[1ex]
&\leq &\mathcal{H}^0(B_k)\cdot \frac{1}{\delta(n_k) (\log n_k)^4} \delta(n_k) \nonumber\\[1ex]
&\leq& \frac{\RDPM_{n_k}((a,b),u_k)}{(\log n_k)^3 \cdot \log(1+(\log n_k)^{-2})},\label{estimate_Bk}
\end{eqnarray}
and we observe that the right-hand side goes to zero as $k\to+\infty$.

Moreover, we have that
\begin{eqnarray}
\sum_{z \in C_k} \left|u_{k}((z+1)\delta(n_h))- u_{k}(z\delta(n_k))\right|
&=& \sum_{z \in C_k} \left| D^{\delta(n_k)}u_k (z\delta(n_k))\right| \delta(n_k) \nonumber\\[1ex]
&\leq &\mathcal{H}^0(C_k) \cdot \frac{\delta(n_k)}{\log n_k} \nonumber\\[1ex]
&\leq& \frac{b_{\delta(n_k)}^*-a_{\delta(n_k),*}}{\delta(n_k)}\cdot \frac{\delta(n_k)}{\log n_k},\label{estimate_Ck}
\end{eqnarray}
and again the right-hand side goes to zero as $k\to +\infty$.

Now we observe that
\begin{align*}
u_k(x)-v_k(x)&=\sum_{z \in \mZ_{\delta(n_k)}(a,b)} \left(u_k((z+1)\delta(n_k))- u_k(z\delta(n_k))\right)\mathbbm{1}_{[(z+1)\delta(n_k),b)}(x)\\
&\quad - \sum_{z \in A_k} \left(u_k((z+1)\delta(n_k))-u_k(z\delta(n_k))\right)\mathbbm{1}_{[(z+1)\delta(n_k),b)}(x), &\forall x \in (a,b)
\end{align*}

Therefore we deduce that for every $x \in (a,b)$ we have that
$$|u_k(x)-v_k(x)|\leq \sum_{z \in B_k \cup C_k} |u_k((z+1)\delta(n_k))- u_k(z\delta(n_k))|,$$
and the right-hand side goes to zero as $k\to+\infty$ thanks to (\ref{estimate_Bk}) and (\ref{estimate_Ck}). This implies that actually $u_k-v_k \to 0$ uniformly, and this concludes the proof of the statement (1).

We point out that the argument used in (\ref{est:RDPM_An}) yields
\begin{equation}
\liminf_{k\to+\infty} \RDPM_{n_k}((a,b),u_k) \geq \liminf_{k\to+\infty} \RDPM_{n_k}((a,b),v_k) \geq \frac{4}{3}m \geq \frac{4}{3}\J((a,b),u),\label{est:RDPM_4m}
\end{equation}
for every sequence $\{u_k\}$ such that eventually $\mathcal{H}^0(A_k)=m$.

\paragraph{\textmd{\textit{Statement (2)}}}

We now focus on the liminf inequality (\ref{liminf_RDPM}). Without loss of generality we can assume that the liminf is finite, so we can find a sequence $\{n_k\}$ of integers such that $n_k \to +\infty$ and
$$\liminf_{n\to +\infty} \RDPM_\delta((a,b),u_n)=\lim_{k\to +\infty} \RDPM_{n_k}((a,b),u_{n_k})<+\infty.$$

Now let us set $T:=\|u\|_\infty +1$, and let us consider the functions
$$w_k(x):=\min\{T,\max\{-T,u_{n_k}(x)\}\}.$$

Then $\|w_k\|_\infty$ is uniformly bounded by $T$, and we have that
$$\RDPM_{n_k}((a,b),u_{n_k})\geq \RDPM_{n_k}((a,b),w_k), $$
and that $w_k\to u$ in $L^p((a,b))$, because $|w_k(x)-u(x)|\leq|u_{n_k}(x)-u(x)|$ for every $x \in (a,b)$.

Therefore from statement (1) and (\ref{est:RDPM_4m}) we deduce that, up to the extraction of a subsequence, it holds
$$\lim_{k\to +\infty} \RDPM_{n_k}((a,b),u_{n_k}) \geq \liminf_{k\to +\infty} \RDPM_{n_k}((a,b),w_{k}) \geq \J((a,b),u),$$
which implies (\ref{liminf_RDPM}).

\paragraph{\textmd{\textit{Statement (3)}}}

In order to prove the limsup inequality (\ref{limsup_RDPM}), we first observe that if $u\in S((a,b))$ is constant, a recovery sequence without boundary conditions is just given by $u_n=u$. If $u\in S((a,b))$ is not constant, the limsup inequality is an immediate consequence of the stronger statement~(4), that we prove below.

\paragraph{\textmd{\textit{Statement (4)}}}

Let $a<x^1<\dots<x^m<b$ be the jump points of $u$ and $u^0,u^1,\dots,u^{m}$ be the values of $u$ in the $m+1$ intervals $(a,x^1),(x^1,x^2),\dots,(x^m,b)$. Since $u\in S((a,b))$ is not constant, it has at least one jump point, namely $m\geq 1$.

For $i\in \{1,\dots,m\}$ and $n\geq 2$ let us set $x^i_n :=\delta(n)\lfloor x^i /\delta(n)\rfloor$. Now we consider the functions
$$u_n(x):=\begin{cases}
A_n &\mbox{if }x \in [a_{\delta(n),*},x^1_n),\\
u^{i} &\mbox{if }x \in [x^i_n,x^{i+1}_n) \mbox{ for some }i \in \{1,\dots,m-1\},\\
B_n &\mbox{if }x \in [x^m_n,b_{\delta(n)}^*]. 
\end{cases}$$

We observe that $u_n \in PC_{\delta(n)}(a,b)$ and
$$\int_a ^b |u_n(x)-u(x)|^p\,dx\leq |A_n-u(a)|^p(x^1_n-a)+\sum_{i=1}^m 2\|u\|_\infty ^p (x^i-x^i_n) + |B_n-u(b)|^p (b-x^m _n),$$
hence $u_n \to u$ in $L^p((a,b))$.

Moreover, we have that
\begin{eqnarray*}
\lim_{n \to +\infty}\RDPM_n ((a,b),u_n)&=&\lim_{n \to +\infty} \ \frac{\delta(n)}{\omega(n)^2}\Biggl[\log\left( 1+\frac{(u^1-A_n)^2}{\delta(n)^2}\right)\\
& &\mbox{}+\sum_{i=1} ^{m-2}\log\left( 1+\frac{(u^{i+1}-u^i)^2}{\delta(n)^2}\right)
+\log\left( 1+\frac{(B_n-u^{m-1})^2}{\delta(n)^2}\right)\Biggr]\\
&=& \frac{4}{3}m\\
&=& \frac{4}{3}\J((a,b),u).
\end{eqnarray*}

If $u$ is constant and the boundary conditions $\{A_n\}$ and $\{B_n\}$ are different, we can not approximate $u$ with constant functions. The best we can do is to fix $z_n \in \mZ_{\delta(n)}(a,b)\setminus \{\lfloor a/\delta(n)\rfloor\}$ and consider the functions
$$u_n(x):=\begin{cases}
A_n &\mbox{if }x \in [a_{\delta(n),*},z_n\delta(n)),\\
B_n &\mbox{if }x \in [z_n\delta(n),b_{\delta(n)}^*]. 
\end{cases}$$

What we get is that
$$\limsup_{n \to +\infty}\RDPM_n ((a,b),u_n)=\limsup_{n\to +\infty} \frac{\delta(n)}{\omega(n)^2} \log\left(1+\frac{(B_n- A_n)^2}{\delta(n)^2}\right) \leq \frac{4}{3}.$$

\paragraph{\textmd{\textit{Statement (5)}}}
 Let us set $T=\|u\|_\infty +1$ and
$$w_k(x):=\min\{T,\max\{-T,u_k(x)\}\}.$$

We observe that $\|w_k\|_\infty$ is bounded by $T$, that $w_k\to u$ in $L^p((a,b))$ and that
\begin{eqnarray*}
\frac{4}{3}\J((a,b),u)&=&\lim_{k\to +\infty}\RDPM_{n_k}((a,b),u_k)\\
&\geq& \limsup_{k\to +\infty} \RDPM_{n_k}((a,b),w_k)\\
&\geq& \liminf_{k\to +\infty} \RDPM_{n_k}((a,b),w_k)\\
&\geq& \frac{4}{3}\J((a,b),u),\end{eqnarray*}
where the last inequality follows from statement (2)

Now $\{w_k\}$ fits into the framework of statement (1), so we can repeat the arguments used in the proof of statement (1), with $w_k$ in place of $u_k$ and the additional information that
$$\lim_{n\to +\infty}\RDPM_{n_k}((a,b),w_k)=\frac{4}{3}\J((a,b),u).$$

In particular, since (\ref{est:RDPM_4m}) becomes a chain of equalities, we deduce that for any subsequence $\{k_h\}$ such that $\mathcal{H}^0(A_{k_h})=m$, we actually have $\J((a,b),u)=m$, hence $\mathcal{H}^0 (A_k)=\J((a,b),u)$ eventually for large $k$.

Moreover, from the expression for the limit (\ref{def:u_limit}) and $\J((a,b),u)=m$, we deduce that $a<x^1<\dots<x^m<b$, and that the total variation of $u$ is given by
$$|Du|((a,b))= \sum_{i=1}^m |J^i|.$$

On the other hand the total variation of $w_k$ is
\begin{eqnarray*}
|Dw_k|((a,b))&=&\sum_{z \in \mZ_{\delta(n_k)} (a,b)} \left|w_k((z+1)\delta(n_k))- w_n(z\delta(n_k))\right| \\
&=& \sum_{i=1} ^m |J^i _k| +\sum_{z \in B_k \cup C_k} \left|w_k((z+1)\delta(n_k))- w_k(z\delta(n_k))\right|,
\end{eqnarray*}
where the numbers $J_k ^i$ and the sets $A_k$, $B_k$ and $C_k$ are defined as in the proof of statement~(1), but with $w_k$ in place of $u_k$.

The last sum tends to zero as $k\to +\infty$ thanks to (\ref{estimate_Bk}) and (\ref{estimate_Ck}), hence we have
$$\lim_{k\to+\infty} |Dw_k|((a,b)) =\sum_{i=1}^m |J^i|=|Du|((a,b)),$$
which means that $w_k	\auto u$ strictly in $BV((a,b))$.

Now from the strict convergence of $\{w_k\}$ and Remark~\ref{rmk:strict} we deduce that for every sequence $\{x_k\}\subset (a,b)$ such that $x_k\to x \in [a,b]$ we have that
\begin{equation}
\liminf_{y\to x} u(y)\leq \liminf_{k\to +\infty} w_k(x_k) \leq \limsup_{k\to +\infty} w_k(x_k) \leq \limsup_{y\to x} u(y).\label{liminf-limsup-strict}
\end{equation}

This implies that $|w_k(x)|<T$ for every $x \in [a_{\delta(n_k),*},b_{\delta(n_k)}^*]$ if $k$ is sufficiently large, and hence $w_k=u_k$ eventually. Therefore it holds also that $u_k\auto u$.

Now we prove the second part of statement (5). To this end, let us assume by contradiction that there exists a subsequence (not relabelled) and a sequence of points $\{x_k\}\subset (a,b)$ such that $u_{k}(x_k)\to \gamma$ for some $\gamma\in \R$ that does not belong to the image of $u$.

Up to a further subsequence, we can assume that $x_k\to x \in [a,b]$, so from (\ref{liminf-limsup-strict}) we deduce that
\begin{equation}\label{liminf<gamma<limsup}
\liminf_{y\to x} u(y)<\gamma<\limsup_{y\to x} u(y),
\end{equation}
and in particular $x\in (a,b)$ is one of the jump points of $u$.

Let $\eta>0$ be such that $[x-\eta,x+\eta]\subset (a,b)$ and $x$ is the only jump point of $u$ in $[x-\eta,x+\eta]$, so in particular
\begin{equation}\label{xi_only_jump}
\{u(x-\eta),u(x+\eta)\}=\left\{\liminf_{y\to x} u(y),\limsup_{y\to x} u(y)\right\}.
\end{equation}

Then, at least for $k$ sufficiently large, we know that $A_{k}\cap \mZ_{\delta(n_k)}(x-\eta,x+\eta) =\{z^i_{k}\} $.

Therefore we have that
\begin{align*}
u_{k}(x_k)&= u_{k}(x-\eta) + \sum_{z \in \mZ_{\delta(n_k)}(x-\eta,x+\eta)} \left(u_{k}((z+1)\delta(n_k))-u_{k}(z\delta(n_k))\right) \mathbbm{1}_{[(z+1)\delta(n_k),b_{\delta(n_k)}^*]}(x_k)\\[0.5ex]
&=u_{k}(x-\eta) + \left(u_k((z^i _k +1)\delta(n_k))-u_{k}(z^i _k\delta(n_k))\right) \mathbbm{1}_{[(z^i _k +1)\delta(n_k),b_{\delta(n_k)}^*]}(x_k) \\[0.5ex]
&\quad\mbox{}+ \sum_{z \in \mZ_{\delta(n_k)}(x-\eta,x+\eta)\setminus\{z^i_k\}} \left(u_{k}((z+1)\delta(n_k))-u_{k}(z\delta(n_k))\right) \mathbbm{1}_{[(z+1)\delta(n_k),b_{\delta(n_k)}^*]}(x_k),
\end{align*}
and the last sum goes to zero because of (\ref{estimate_Bk}) and (\ref{estimate_Ck}). Hence the only possible limits for $u_k(x_k)$ are $u(x-\eta)$ and $u(x-\eta)+J^i=u(x+\eta)$.

In any case, this is a contradiction with (\ref{liminf<gamma<limsup}) and (\ref{xi_only_jump}).

\paragraph{\textmd{\textit{Statement (6)}}}

The argument is basically the same used to prove statement (5), but in this case from (\ref{est:RDPM_4m}) we deduce only that eventually $\mathcal{H}^0(A_k) \in\{0,1\}$.

If we consider a subsequence such that $A_k=\emptyset$, then from (\ref{estimate_Bk}) and (\ref{estimate_Ck}) we obtain that the convergence is strict on $(a,b)$.

So let us consider a subsequence such that $\mathcal{H}^0(A_k)=1$. Since $u$ is constant, from (\ref{def:u_limit}) we deduce that either $J^1=0$ or $x^1\in \{a,b\}$.

In the first case, we have again strict convergence on $(a,b)$ thanks to (\ref{estimate_Bk}), (\ref{estimate_Ck}) and $J^1_k \to 0$, while in the second case we have strict convergence on every subinterval $(c,d)$, because eventually $x^i _k \notin (c,d)$.
\end{proof}

\subsection{Minimum problems with fidelity term}

We can now add the fidelity term to the functionals (\ref{def:RDPM_n}), so we obtain the functionals
\begin{equation}\label{def:RDPMF_n}
\RDPMF_n(\beta,f,(a,b),u):=\begin{cases}
\displaystyle{\RDPM_n((a,b),u) + \beta\int_a ^b (u(x)-f(x))^2\,dx} &\mbox{if }u\in PC_{\delta(n)}(a,b),\\
+\infty &\mbox{otherwise}.
\end{cases}
\end{equation}

Since the fidelity term is continuous with respect to the metric of $L^2((a,b))$, we deduce that the $\Gamma$-limit of (\ref{def:RDPMF_n}) with respect to the $L^2$ convergence is the functional
\begin{equation}\label{def:JF}
\JF(\alpha,\beta,f,(a,b),u)=\alpha\J(\Omega,u)+\beta \int_\Omega (u(x)-f(x))^2\,dx,
\end{equation}
with $\alpha=4/3$.

Now we restrict ourselves to the case in which $(a,b)=(0,L)$ for some $L>0$ and $f(x)=Mx$ for some $M\in \R$, and we consider the following minimum problems without boundary conditions
\begin{gather}
\mu_n(\beta,L,M):=\min \left\{\RDPMF_n(\beta,Mx,(0,L),u):u \in PC_{\delta(n)}(0,L)\right\},\label{def:mu_n}\\
\mu(\alpha,\beta,L,M):=\min \left\{ \JF(\alpha,\beta,Mx,(0,L),u):u \in S((0,L))\right\},\label{def:mu}
\end{gather}
and the following minimum problems with boundary conditions
\begin{gather}
\mu_n ^*(\beta,L,M):=\min \left\{\RDPMF_n(\beta,Mx,(0,L),u):u \in PC_{\delta(n)}(0,L),u(0)=0,u(L)=ML\right\},\label{def:mu_n*}\\
\mu^*(\alpha,\beta,L,M):=\min \left\{ \JF(\alpha,\beta,Mx,(0,L),u):u \in S((0,L)),u(0)=0,u(L)=ML\right\}.\label{def:mu*}
\end{gather}

In the following result we list some properties of these minimum problems that we need in the sequel.

\begin{prop}\label{prop:mu_mu*}
The minimum problems in (\ref{def:mu_n}) through (\ref{def:mu*}) have the following properties.
\begin{enumerate}
\renewcommand{\labelenumi}{(\arabic{enumi})}

\item There exist minimizers for (\ref{def:mu_n}), (\ref{def:mu}) and (\ref{def:mu*}) for every choice of the parameters $(\alpha,\beta,L,M) \in \times (0,+\infty)^3 \times \R$ and for every $n\geq 2$. For (\ref{def:mu_n*}) the same holds provided that $L>\delta(n)$.

\item For every admissible choice of $n,\alpha,\beta,L$ the functions
\begin{gather*}
M\mapsto \mu_n(\beta,L,M), \qquad M	\mapsto \mu(\alpha,\beta,L,M),\\
M\mapsto \mu_n ^*(\beta,L,M), \qquad M	\mapsto \mu^*(\alpha,\beta,L,M)
\end{gather*}
are even, continuous in $\R$ and nondecreasing in $[0,+\infty)$.

Moreover, it turns out that
\begin{equation}\label{est:mu_delta*}
\mu_n ^* (\beta,L,M_2)\leq\left(\frac{M_2}{M_1}\right)^2\mu_n ^* (\beta,L,M_1)\end{equation}
for every $M_2\geq M_1 >0$.

\item For every admissible choice of $n,\alpha,\beta,M$ the functions
$$L\mapsto \mu_n(\beta,L,M), \qquad L	\mapsto \mu(\alpha,\beta,L,M), \qquad L	\mapsto \mu^*(\alpha,\beta,L,M)$$
are nondecreasing in $(0,+\infty)$.

As for $\mu_n^*(\beta,L,M)$, it turns out that
\begin{equation}\label{eq:monot_mu_n*}
\mu_{n}^* \left(\beta,L_{\delta(n)}^*, M\right)\leq \mu_{n}^*(\beta,L,M)+\frac{\log(2M^2+2)}{\log n} + \frac{\beta M^2 \delta(n)^3}{3}
\end{equation}
for every $L>\delta(n)$, where $L_{\delta(n)}^*$ is defined according to (\ref{def:a_delta-b_delta}).

\item For every admissible choice of $\beta,L,M$ it turns out that
$$\lim_{n\to +\infty} \mu_n(\beta,L,M)=\mu(4/3,\beta,L,M),\qquad\qquad \lim_{n\to +\infty} \mu_n ^*(\beta,L,M)=\mu^*(4/3,\beta,L,M).$$

\item For every admissible choice of $\beta,L$ it turns out that
\begin{gather*}
\lim_{n\to +\infty}\sup_{|M|\leq M_0} |\mu_n(\beta,L,M)-\mu(4/3,\beta,L,M)|=0 \qquad \forall M_0>0,\\
\lim_{n\to +\infty}\sup_{|M|\leq M_0} |\mu_n ^*(\beta,L,M)-\mu^*(4/3,\beta,L,M)|=0 \qquad \forall M_0>0.
\end{gather*}
\end{enumerate}
\end{prop}

\begin{proof}
The arguments are quite standard, and similar to those used in the proof of \cite[Proposition 3.4]{GP:fastPM-CdV}. The main difference is that discrete functions cannot be rescaled horizontally because that would alter the length of the steps. For this reason Statement~(3) requires a bit more work.

\paragraph{\textmd{\textit{Statement (1)}}} The existence of minimizers for $\mu_n$ and $\mu_n ^*$ follows from the coercivity of the fidelity term, because the space $PC_{\delta(n)}(0,L)$ is finite dimensional and the functional is continuous. In the case of $\mu_n ^*$ we also need that $L>\delta(n)$, otherwise the space of functions in $PC_{\delta(n)}(0,L)$ satisfying the boundary conditions is empty.

For $\mu$, it is just a simple application of the direct method in the calculus of variations, since the functional $\JF$ is coercive and lower semicontinuous with respect to the weak BV topology.

The problem is slightly less trivial for $\mu^*$, since the boundary conditions do not pass to the limit. However, it is enough to relax the boundary conditions by considering the following functional without boundary conditions
$$\JF(\alpha,\beta,Mx,(0,L),u)+\alpha\mathbbm{1}_{\R\setminus\{0\}}(u(0))+\alpha\mathbbm{1}_{\R\setminus\{ML\}}(u(L)).$$

The existence of a minimizer $v\in S((0,L))$ follows now from the direct method. Then we can prove that $v$ verifies the boundary conditions by comparing the value of the functional with a competitor $v_\tau$ that is equal to $v$ in $(\tau,L-\tau)$, equal to $0$ in $(0,\tau)$ and to $ML$ in $(L-\tau,\tau)$, where $\tau \in (0,L/2)$. When $\tau$ is small enough, it may be seen that if $v$ does not respect the boundary conditions, then $v_\tau$ contradicts the minimality of $v$.

\paragraph{\textmd{\textit{Statement (2)}}} Symmetry and continuity are trivial. As for the monotonicity, let us fix $M_2>M_1>0$ and let $u_2$ be a minimizer for $\mu_n (\beta,L,M_2)$. Let us set $u_1=(M_1/M_2)u_2$. Then it turns out that
$$\RDPM_n ((0,L),u_1)\leq \RDPM_n ((0,L),u_2)$$
and
$$\int_0 ^L (u_1-M_1x)^2\,dx = \left(\frac{M_1}{M_2}\right)^2 \int_0 ^L (u_2-M_2x)^2\,dx \leq \int_0 ^L (u_2-M_2x)^2\,dx.$$

Therefore we have that
$$\mu_n(\beta,L,M_1)\leq \RDPMF_n(\beta,M_1 x,(0,L),u_1)\leq \RDPMF_n(\beta,M_2 x,(0,L),u_2) =\mu_n(\beta,L,M_2).$$

The same argument works also for the monotonicity of $\mu$, $\mu_n ^*$ and $\mu^*$.

In order to prove (\ref{est:mu_delta*}), we reverse the argument, namely we let $u_1$ be a minimizer for $\mu_n ^*(\beta,L,M_1)$ and we set $u_2=(M_2/M_1)u_1$. It follows that
\begin{eqnarray*}
\RDPM_{n}((0,L),u_2)&=&\frac{1}{\omega(n) ^2}\int_{0} ^{L_{\delta(n)}^*-\delta(n)} \log\left(1+\left(\frac{M_2}{M_1}\right)^2 D^{\delta(n)}u_1(x)^2\right) dx \\
&\leq& \left(\frac{M_2}{M_1}\right)^2 \RDPM_{n}((0,L),u_1),
\end{eqnarray*}
where the inequality follows from the fact that
$$\log(1+\lambda t)\leq \lambda \log(1+t) \qquad \forall \lambda \geq 1 \ \forall t \geq 0.$$

Since the fidelity terms scales as above and $u_2$ is a competitor for the minimum problem $\mu_n ^*(\beta,L,M_2)$, we deduce that
\begin{eqnarray*}
\mu_n ^* (\beta,L,M_2) &\leq& \RDPMF_n(\beta,M_2 x,(0,L),u_2)\\
&\leq& \left(\frac{M_2}{M_1}\right)^2 \RDPMF_n(\beta,M_1 x,(0,L),u_1) \\
&=&\left(\frac{M_2}{M_1}\right)^2 \mu_n ^* (\beta,L,M_1).
\end{eqnarray*}

\paragraph{\textmd{\textit{Statement (3)}}} Let us fix $L_2 \geq L_1>0$. In order to prove the monotonicity of $\mu$ and $\mu_n$ it is enough to consider the restriction to $(0,L_1)$ and $(0,(L_1)_{\delta(n)}^*)$ of the minimizers on $(0,L_2)$ and $(0,(L_2)_{\delta(n)}^*)$.

As for $\mu^*$, since we have to take into account the boundary conditions, we consider the function $u_1(x)=(L_1/L_2)u_2(L_2x/L_1)$, where $u_2$ is a minimizer for $\mu^*(\alpha,\beta,L_2,M)$. Then it turns out that $\J((0,L_1),u_1)=\J((0,L_2),u_2)$, while
\begin{equation}\label{eq:L2scale_u1_u2}
\int_0 ^{L_1} (u_1(x)-Mx)^2\,dx = \left(\frac{L_1}{L_2}\right)^3 \int_0 ^{L_2} (u_2(x)-Mx)^2\,dx,
\end{equation}
hence
$$\mu^*(\alpha,\beta,L_1,M)\leq \JF(\alpha,\beta,Mx,(0,L_1),u_1)\leq \JF(\alpha,\beta,Mx,(0,L_2),u_2) =\mu^*(\alpha,\beta,L_2,M).$$

For $\mu_n ^*$ none of the previous strategies work, because we have to take into account both the boundary conditions and the length of the steps in the definition of the space $PC_{\delta(n)}$.

So we let $u_1\in PC_{\delta(n)}(0,L)$ be a minimizer for $\mu_n ^*(\beta,L,M)$ and we consider the function $u_2 \in PC_{\delta(n)}(0,L)$ that coincides with $u_1$ in $[0,L_{\delta(n)}^*-\delta(n))$, and is equal to $M L_{\delta(n)}^*$ in $[L_{\delta(n)}^*-\delta(n),L_{\delta(n)}^*]$. We observe that $u_2$ is an admissible competitor for the minimum problem $\mu_n ^*(\beta,L_{\delta(n)}^*,M)$, so we deduce that
\begin{align}
\mu_n^*(\beta,L_{\delta(n)}^*,M)&\leq \RDPMF_n(\beta,Mx,(0,L_{\delta(n)}^*),u_2)\nonumber\\
&=\RDPMF_n(\beta,Mx,(0,L),u_1)\nonumber\\
&\quad\mbox{} +\frac{\delta(n)}{\omega(n)^2}\left[ \log\left(1+\frac{(ML_{\delta(n)}^*-A)^2}{\delta(n)^2}\right)-\log\left(1+\frac{(ML-A)^2}{\delta(n)^2}\right)\right]\nonumber\\
&\quad\mbox{} +\beta \int_{L_{\delta(n)}^*-\delta(n)} ^{L_{\delta(n)}^*} (ML_{\delta(n)}^*-Mx)^2\,dx -\beta\int_{L_{\delta(n)}^*-\delta(n)} ^L(ML-Mx)^2\,dx, \label{est:mu_n*-L_delta(n)}
\end{align}
where $A=u_1(L_{\delta(n)}^*-2\delta(n))$.

The second line in the right-hand side is equal to
$$\frac{1}{\log n}\log\left(\frac{\delta(n)^2 + (ML_{\delta(n)}^*-A)^2}{\delta(n)^2 + (ML-A)^2}\right).$$

Since $L_{\delta(n)}^* \leq L+\delta(n)$, we can estimate the previous expression in the following way
\begin{eqnarray*}
\frac{1}{\log n}\log\left(\frac{\delta(n)^2 + (ML_{\delta(n)}^*-A)^2}{\delta(n)^2 + (ML-A)^2}\right)&\leq& \frac{1}{\log n}\log\left(\frac{(2M^2+1)\delta(n)^2 + 2(ML-A)^2}{\delta(n)^2 + (ML-A)^2}\right)\\
&\leq& \frac{1}{\log n} \log(\max\{2M^2+1 , 2\})\\
&\leq& \frac{\log(2M^2+2)}{\log n}.
\end{eqnarray*}

Finally, we can estimate the last line of (\ref{est:mu_n*-L_delta(n)}) simply neglecting the negative addendum, while computing the positive one we obtain
$$\int_{L_{\delta(n)}^*-\delta(n)} ^{L_{\delta(n)}^*} (ML_{\delta(n)}^*-Mx)^2\,dx =\frac{\beta M^2 \delta(n)^3}{3}.$$

Plugging these estimates into (\ref{est:mu_n*-L_delta(n)}) we get exactly (\ref{eq:monot_mu_n*}).

%we set again $u_1(x)=(L_1/L_2)u_2(L_2x/L_1)$, where $u_2$ is now a minimizer for $\mu_{\delta\sqrt{L_2/L_1}} ^*(\beta,L_2,M)$. We observe that $u_1 \in PC_{\delta^2}(0,L_1)$, because $u_2\in PC_{\delta^2(L_2/L_1)}(0,L_2)$.

%The fidelity terms scales as in (\ref{eq:L2scale_u1_u2}), while the functional $\RDPM_\delta$ satisfies
%\begin{eqnarray*}
%\RDPM_{\delta}((0,L_1),u_1)&=&\frac{1}{\delta^2 |\log(\delta)|}\int_{0} ^{(L_1)_{\delta^2}} \log\left(1+D^{\delta^2}u_1(x)^2\right) dx\\
%&=&\frac{L_1/L_2}{\delta^2 |\log(\delta)|}\int_{0} ^{(L_2)_{\delta^2(L_2/L_1)}} \log\left(1+D^{\delta^2(L_2/L_1)}u_2(x)^2\right) dx\\
%&\leq& \RDPM_{\delta\sqrt{L_2/L_1}}((0,L_2),u_2),
%\end{eqnarray*}

%Therefore we deduce that
%\begin{eqnarray*}
%\mu_\delta ^* (\beta,L_1,M)&\leq& \RDPMF_{\delta}(\beta,Mx,(0,L_1),u_1)\\
%&\leq& \RDPMF_{\delta\sqrt{L_2/L_1}}(\beta,Mx,(0,L_2),u_2)\\
%&=&\mu_{\delta\sqrt{L_2/L_1}} ^*(\beta,L_2,M),
%\end{eqnarray*}
%hence (\ref{eq:monot_mu_delta*}) holds.

\paragraph{\textmd{\textit{Statement (4)}}} This is a consequence of Theorem~\ref{thm:gamma-convergence}.

\paragraph{\textmd{\textit{Statement (5)}}} The uniformity of the limit on bounded subsets follows from the pointwise convergence and the symmetry, continuity and monotonicity properties in Statement~(2).
\end{proof}

\setcounter{equation}{0}
\section{Convergence of local minimizers}

The aim of this section is to prove that local minimizers of (\ref{def:RDPMF_n}) defined on intervals invading the real line converge to entire local minimizers for the functional (\ref{def:JF}), and to classify the latter.

Before proceeding, we want to make clear what we mean by local minimizers.

\begin{defn}
Let $(a,b)$ be an interval and let us fix a positive integer $n \geq 2$, a positive constant $\beta >0$ and a function $f \in L^2((a,b))$. We say that a function $v \in PC_{\delta(n)}(a,b)$ is a local minimizer for $\RDPMF_n(\beta,f,(a,b),\cdot)$ if
$$\RDPMF_n(\beta,f,(a,b),v)\leq \RDPMF_n(\beta,f,(a,b),u)$$
for every $u \in PC_{\delta(n)}(a,b)$ with the same boundary values of $v$, namely such that $u(a)=v(a)$ and $u(b)=v(b)$.
\end{defn}

\begin{defn}
Let $(a,b)$ be an interval and let us fix two positive constants $\alpha,\beta >0$ and a function $f \in L^2((a,b))$. We say that a function $v \in S((a,b))$ is a local minimizer for $\JF(\alpha,\beta,f,(a,b),\cdot)$ if
$$\JF(\alpha,\beta,f,(a,b),v)\leq \JF(\alpha,\beta,f,(a,b),u)$$
for every $u \in S((a,b))$ with the same boundary values of $v$, namely such that $u(a)=v(a)$ and $u(b)=v(b)$.

We also say that $v\in S_{loc}(\R)$ is an entire local minimizer for $\JF(\alpha,\beta,f,\cdot,\cdot)$ if its restriction to every bounded interval $(a,b)$ is a local minimizer for $\JF(\alpha,\beta,f,(a,b),\cdot)$.
\end{defn}

We can now state the main results of this section. The first one is an estimate on the minimum values $\mu$ and $\mu^*$ defined in (\ref{def:mu}) and (\ref{def:mu*}), which is fundamental in the proof of Theorem~\ref{thm:min}.

\begin{prop}\label{prop:estimate_mu}
For every $\alpha,\beta,L,M\in (0,+\infty)^3\times \R$ it holds that
\begin{align}
\frac{1}{2} \left(\frac{9\alpha^2 \beta M^2}{2}\right)^{1/3}L -2\cdot 6^{2/3}\alpha  &\leq \mu(\alpha,\beta,L,M)\label{eq:estimate_mu_below} \\
&\leq \mu^*(\alpha,\beta,L,M) \leq \frac{1}{2} \left(\frac{9\alpha^2 \beta M^2}{2}\right)^{1/3}L + \frac{3\alpha}{2},\label{eq:estimate_mu_above}
\end{align}
\end{prop}

The next result is the characterization of entire local minimizers for $\JF(\alpha,\beta,Mx,\cdot,\cdot)$.

\begin{prop}\label{prop:local_minimizers}
Let us fix $(\alpha,\beta,M)\in (0,+\infty)^2\times \R$ and let us consider the canonical $(H,V)$-staircase $S_{H,V}$ with parameters given by
\begin{equation}\label{def:H_V}
H:=\frac{1}{2}\left(\frac{6\alpha}{\beta M^2}\right)^{1/3},\qquad V:=MH,\end{equation}
and the understanding that $S_{H,V}\equiv 0$ when $M=0$.

Then the set of entire local minimizers for the functional $\JF(\alpha,\beta,Mx,\cdot,\cdot)$ coincides with the set $\obl(H,V)$ of oblique translations of $S_{H,V}$, introduced in Definition~\ref{defn:translations}.
\end{prop}

Before proving Proposition~\ref{prop:estimate_mu} and Proposition~\ref{prop:local_minimizers}, we state some properties of local minimizers for $\JF(\alpha,\beta,Mx,(a,b),\cdot)$. We do not include a complete proof because the arguments are the same used in \cite[Section~6.2]{GP:fastPM-CdV}, and actually the computations here would be even easier because small perturbations of the position and the height of the jumps do not affect the value of $\J$, but only the fidelity term.

\begin{lemma}\label{lemma:prop_minimizers}
Let us fix $(\alpha,\beta,M)\in (0,+\infty)^2\times \R$ and let $(a,b)$ be an interval. Let $v$ be a local minimizer for the functional $\JF(\alpha,\beta,Mx,(a,b),\cdot)$ and, if $M\neq 0$, let us set
\begin{equation}\label{def:L_0}
L_0:=2\left(\frac{\alpha}{\beta M^2}\right)^{1/3}.
\end{equation}

Then $v$ has the following properties.
\begin{enumerate}
\renewcommand{\labelenumi}{(\arabic{enumi})}

\item If $M\neq 0$ and $b-a>L_0$ then in $(a,b)$ there exists either a jump point $x \in S_v$ or an intersection of the function $v$ with the line $Mx$, namely a point $y \in (a,b)\setminus S_v$ such that $v(y)=My$.

\item If $x\in S_v$ is a jump point of $v$, then $v(x^+)-Mx=Mx-v(x^-)$ and this value has the same sign of $M$ if $M \neq 0$.

\item If $x_1<x_2$ are two consecutive jump points of $v$, namely $v$ has no other jump points in $(x_1,x_2)$, then $v(x)=M(x_1+x_2)/2$ for every $x \in (x_1,x_2)$.

\item If $M\neq 0$ and $y_1<y_2<\dots<y_m$ are the intersection points of the function $v$ with the line $Mx$ in $(a,b)$, then $y_2-y_1= y_3-y_2= \dots=y_m-y_{m-1}$.
\end{enumerate}
\end{lemma}

\begin{proof}
Statement (2) can be proved exactly as in \cite{GP:fastPM-CdV}, just by considering horizontal variations, while statement (3) follows considering vertical variations. In both cases it is enough to minimize the cost of the fidelity term, since these variations do not affect the number of jump points. Statement (4) follows immediately from statement (2) and statement (3).

Statement (1) can also be proved with the same argument used in \cite{GP:fastPM-CdV}, however we include the proof in order to show the computation of the value $L_0$.

%[Forse ci vorrebbe di meno a fare le dimostrazioni]

So let us assume by contradiction that $v$ has neither jump points nor intersections with $Mx$ in $(a,b)$. It follows that $v(x)\equiv c$, for some constant $c \in (-\infty,Ma)\cup (Mb,\infty)$. Hence we have that
$$\JF(\alpha,\beta,Mx,(a,b),v)=\beta \int_a ^b (Mx-c)^2\,dx \geq \frac{\beta M^2}{3}(b-a)^3.$$

Now, for $\tau \in (0,(b-a)/2)$ let us consider the function
$$u_\tau(x)=\begin{cases}
c &\mbox{if }x \in (a,a+\tau),\\
(b-a)/2 &\mbox{if }x \in (a+\tau,b-\tau),\\
c &\mbox{if }x \in (b-\tau,b).
\end{cases}$$

Since $v$ is a local minimizer, we know that $\JF(\alpha,\beta,Mx,(a,b),v)\leq \JF(\alpha,\beta,Mx,(a,b),u_\tau)$ for every $\tau$, and in particular
$$\frac{\beta M^2}{3}(b-a)^3\leq \JF(\alpha,\beta,Mx,(a,b),v)\leq \lim_{\tau\to 0^+}\JF(\alpha,\beta,Mx,(a,b),u_\tau)=2\alpha+\frac{\beta M^2}{12}(b-a)^3.$$

It follows that $b-a \leq L_0$.
\end{proof}

We can now prove Proposition~\ref{prop:estimate_mu} and Proposition~\ref{prop:local_minimizers}.

\begin{proof}[Proof of Proposition~\ref{prop:estimate_mu}]
First of all, we observe that both the estimates are trivial when $M=0$, so we can assume that $M\neq 0$.

We start with the estimate from below (\ref{eq:estimate_mu_below}). We observe that the estimate is trivial if $L\leq 4L_0$, where $L_0$ is given by (\ref{def:L_0}), because in this case the left-hand side of (\ref{eq:estimate_mu_below}) is negative, so we can assume that $L>4L_0$.

Let $v$ be a minimizer for $\mu(\alpha,\beta,L,M)$. We claim that there exists two intersections points $a_0\in (0,2L_0]$ and $b_0 \in [L-2L_0,L)$ of the function $v$ with the line $Mx$.

Indeed, from Lemma~\ref{lemma:prop_minimizers} we know that for every positive number $\eta \in (0,L/4-L_0)$ the function $v$ has either a jump point or an intersection with the line $Mx$ in each of the four intervals $(0,L_0+\eta)$, $(L_0+\eta,2L_0+2\eta)$, $(L-2L_0-2\eta,L-L_0-\eta)$ and $(L-L_0-\eta,L)$.

Since statement (2) in Lemma~\ref{lemma:prop_minimizers} implies that between two jump points there is necessarily an intersection with the line $Mx$, we deduce that $v$ has such an intersection in both the intervals $(0,2L_0+2\eta)$ and $(L-2L_0-2\eta,L)$ for every $\eta\in (0,L/4-L_0)$. Since the number of intersections is finite, we deduce that our claim is true.

Now from Lemma~\ref{lemma:prop_minimizers} we know that the interval $(a_{0},b_{0})$ is divided into a finite number $m\geq 1$ of intervals of equal length whose endpoints are intersections. Moreover, $v$ has exactly one jump point in the midpoint between any two consecutive intersection. As a consequence, the shape of $v$ in $(a_{0},b_{0})$ depends only on $m$, and with an elementary computation we find that
\begin{eqnarray*}
\mu(\alpha,\beta,L,M)&=&\JF(\alpha,\beta,Mx,(0,L),v)\\
&\geq& \JF(\alpha,\beta,Mx,(a_0,b_0),v)\\
&=& m \left[\alpha + \frac{\beta M^2}{12}\left(\frac{b_0-a_0}{m}\right)^3\right]
\end{eqnarray*}

Therefore, from the inequality
$$A+B \geq 3\left(\frac{A^2 B}{4}\right)^{1/3}\qquad \forall A,B\geq 0,$$
we deduce that
$$\mu(\alpha,\beta,L,M) \geq 3\left(\frac{\alpha^2 \beta M^2}{48}\right)^{1/3}(b_0-a_0) \geq \frac{1}{2}\left(\frac{9\alpha^2 \beta M^2}{2}\right)^{1/3}(L-4L_0),$$
Substituting the value of $L_0$ we obtain exactly (\ref{eq:estimate_mu_below}).

Now we prove the estimate from above (\ref{eq:estimate_mu_above}). To this end, let $H$ be as in (\ref{def:H_V}) and let us set $m=\lceil L/(2H)\rceil$.

Let also $v\in S((0,2mH))$ be the function that intersects the line $Mx$ in $0,2H,\dots,2mH$ and has jumps in the midpoints of the intervals between two consecutive intersections.

Since $v$ is a competitor for $\mu^*(\alpha,\beta,2mH,M)$, from the monotonicity of $\mu^*$ with respect to $L$ we deduce that
\begin{eqnarray}
\mu^*(\alpha,\beta,L,M) &\leq& \mu^*(\alpha,\beta,2mH,M)\nonumber\\
&\leq & \JF(\alpha,\beta,Mx,(0,2mH),v) \nonumber\\
&=& m \left(\alpha + \frac{2\beta M^2 H^3}{3}\right)\nonumber\\
&\leq& \left(\frac{L}{2H}+1\right)\left(\alpha + \frac{2\beta M^2 H^3}{3}\right).\label{eq:estimate_mu*}
\end{eqnarray}

Substituting the value of $H$ given by (\ref{def:H_V}), we obtain exactly (\ref{eq:estimate_mu_above}).
\end{proof}

\begin{proof}[Proof of Proposition~\ref{prop:local_minimizers}]
Let $v\in S_{loc}(\R)$ be an entire local minimizer for $\JF(\alpha,\beta,Mx,\cdot,\cdot)$.

If $M=0$, from statement (2) of Lemma~\ref{lemma:prop_minimizers} we deduce that $|v|$ must be equal to a constant $c\in \R$. If $c\neq 0$, we see that $v$ can not be a local minimizer on large intervals, because the cost of the fidelity term grows linearly with the length of the interval, while a function which vanishes everywhere but in a small neighborhood of the boundary, where it has two jumps in order to attain the boundary conditions of $v$, has a cost that is only slightly larger than $2\alpha$. It follows that the null function is the unique entire local minimizer for $\JF(\alpha,\beta,Mx,\cdot,\cdot)$ when $M=0$.

If $M\neq 0$, from Lemma~\ref{lemma:prop_minimizers} we know that the set of intersection points of $v$ with the line $Mx$ is discrete, and divides the real line in intervals of the same length $2h>0$, while the set of jump points of $v$ consists of the midpoints of these intervals. This means that $v$ is an oblique translation of the canonical $(h,Mh)$-staircase. We claim that necessarily $h=H$.

Up to an oblique translation, we can assume that the intersection points are of the form $2zh$, with $z \in \Z$. Let us consider the interval $(0,2mh)$, where $m$  is a positive integer. Since $v$ is an entire local minimizer and we have $v(0)=0$ and $v(2mh)=2Mmh$, we deduce that
$$\JF(\alpha,\beta,Mx,(0,2mh),v)=\mu^*(\alpha,\beta,2mh,M).$$

Then by (\ref{eq:estimate_mu*}) with $L=2mh$ we get that
\begin{eqnarray*}
m \left(\alpha + \frac{2\beta M^2 h^3}{3}\right) &=& \JF(\alpha,\beta,Mx,(0,2mh),v)\\
&=& \mu^*(\alpha,\beta,2mh,M)\\
&\leq& \left(\frac{2mh}{2H}+1\right)\left(\alpha + \frac{2\beta M^2 H^3}{3}\right).
\end{eqnarray*}

Dividing by $mh$ and letting $m\to +\infty$ we obtain that
$$\frac{\alpha}{h} + \frac{2\beta M^2 h^2}{3} \leq \frac{\alpha}{H} + \frac{2\beta M^2 H^2}{3}.$$

This implies that $h=H$, because $H$ is the unique minimum point of the function
$$h\mapsto \frac{\alpha}{h} + \frac{2\beta M^2 h^2}{3}.$$

We point out that, at this point, we have only proved that if an entire local minimizer exists, then it is an oblique translation of the canonical $(H,V)$-staircase $S_{H,V}$. We still have to prove that an entire local minimizer exists or, equivalently, that these staircases are actually entire local minimizers.

To this end, it is enough to prove that $S_{H,V}$ is a local minimizer on intervals of the form $(2z_1H,2z_2H)$, where $z_1<z_2$ are integer numbers. In this case, the minimum problem with the boundary data given by $S_{H,V}$ coincides (up to a translation) with $\mu^*(\alpha,\beta,2(z_2-z_1)H,M)$, so we know that a minimizer exists by Proposition~\ref{prop:mu_mu*}. By Lemma~\ref{lemma:prop_minimizers} this minimizer is necessarily a staircase with steps of equal length and height, so the only unknown is the number of jump points $m$. We end up with the following minimum problem
$$\min_{m \in \N} \left\{m\left[\alpha + \frac{2\beta M^2}{3} \left(\frac{2H(z_2-z_1)}{2m}\right)^3\right]\right\}=\min_{m \in \N} \left\{m\alpha + \frac{\alpha}{2} \frac{(z_2-z_1)^3}{m^2}\right\},$$
which is solved by $m=z_2-z_1$, that corresponds to the function $S_{H,V}$.
\end{proof}

The last main result of this section is the convergence of minimizers of $\RDPMF_n$ to minimizers of $\JF$, which is the main tool in the proof of Theorem~\ref{thm:blow-up}.

\begin{prop}\label{prop:convergence_minimizers}
Let us fix $\beta>0$ and $M \in \R$. For every $k\in \N$, let $n_k\geq 2$ be an integer and $A_{k}<B_{k}$ be real numbers, let $g_{k}:(A_{k},B_{k})\to\R$ be a continuous function, and let $w_{k}\in PC_{\delta(n_k)}(A_{k},B_{k})$.

Let us assume that
\begin{enumerate}
\renewcommand{\labelenumi}{(\roman{enumi})}
\item  as $k\to +\infty$ it turns out that $n_k\to +\infty$, $A_{k}\to -\infty$, and $B_{k}\to +\infty$,

\item  $g_{k}(x)\to Mx$ uniformly on bounded subsets of $\R$,

\item  for every $k\in \N$ it turns out that the function $w_k$ is a local minimizer for the functional $\RDPMF_{n_k}(\beta,g_k,(A_k,B_k),\cdot)$,

\item  there exists a constant $C>0$ such that
\begin{equation}
\RDPMF_{n_k}(\beta,g_k,(A_k,B_k),w_k)\leq \frac{C}{\omega(n_k)} \qquad \forall k\geq 1.\label{hp:RDPMF<C/delta}
\end{equation}
\end{enumerate}

Then there exists an increasing sequence $\{k_{h}\} \subset \N$ such that
\begin{equation}
w_{k_{h}}\auto w_{\infty}
\quad\text{locally strictly in }BV\loc(\R),
\nonumber
\end{equation}
where $w_{\infty}$ is an entire local minimizer for the functional $\JF(4/3,\beta,Mx,\cdot,\cdot)$.

Moreover, for every compact set $K\subset \R$ such that $w_\infty ^{-1}(K)=\emptyset$ and every positive number $R>0$, it holds that $w_{k_h} ^{-1}(K)\cap (-R,R)=\emptyset$ when $h$ is sufficiently large.
\end{prop}

\begin{proof}
Let us fix $L>0$. We claim that $\{w_k\}$ satisfies
\begin{equation}\label{claim:compactness_w_k}
\limsup_{k\to +\infty}\ \bigl\{\RDPM_{n_k}((-L,L),w_k) + \|w_k\|_{L^\infty((-L,L))} \bigr\} <+\infty.\end{equation}

To prove this, let $k_L \in \N$ be such that $(-L-2,L+2)\subset (A_k,B_k)$ for every $k\geq k_L$ and let us set
$$M_L:=\sup_{k\geq k_L} \|g_k\|_{L^\infty(-L-2,L+2)}.$$

For every $k\geq k_L$ we fix $a_k \in \delta(n_k) \mZ_{\delta(n_k)}(-L-2,-L-1)$ and $b_k \in \delta(n_k) \mZ_{\delta(n_k)}(L+1,L+2)$ so that
\begin{gather*}
|w_k(a_k)|=\min\{|w_k(z\delta(n_k))|:z \in \mZ_{\delta(n_k)}(-L-2,-L-1)\},\\
|w_k(b_k)|=\min\{|w_k(z\delta(n_k))|:z \in \mZ_{\delta(n_k)}(L+1,L+2)\}.
\end{gather*}

Then, by (\ref{hp:RDPMF<C/delta}) and the elementary inequality
$$(A-B)^2\geq \frac{A^2}{2}-B^2 \qquad \forall (A,B)\in \R^2,$$
for every $k\geq k_L$ we get that
\begin{eqnarray}
\frac{C}{\omega(n_k)} &\geq& \RDPMF_{n_k}(\beta,g_k,(A_k,B_k),w_k)\nonumber\\
 &\geq& \beta\int_{-L-2} ^{-L-1} (w_k-g_k)^2\,dx + \beta\int_{L+1} ^{L+2} (w_k-g_k)^2\,dx \nonumber\\
&\geq& \beta\int_{-L-2} ^{-L-1} \left(\frac{w_k^2}{2}-g_k^2\right)dx + \beta\int_{L+1} ^{L+2} \left(\frac{w_k^2}{2}-g_k^2\right)dx \nonumber\\
&\geq& \frac{\beta}{2}\left(w_k(a_k)^2 + w_k(b_k)^2\right) - 2\beta M_L^2,\label{argument_ak_bk}
\end{eqnarray}
namely
\begin{equation}
w_k(a_k)^2 + w_k(b_k)^2 \leq \frac{2C}{\beta\omega(n_k)} + 4M_L ^2. \label{eq:est-w(ak)_w(bk)}
\end{equation}

On the other hand, since $w_k$ is a local minimizer, we can estimate its energy from above simply comparing it with the energy of any other function with the same boundary values. To this end, for every $k\geq k_L$ we consider the function $v_k$ that coincides with $w_k$ on $[a_k,a_k +\delta(n_k))$ and $[b_k,b_k+\delta(n_k))$ and vanishes on $[a_k+\delta(n_k),b_k)$.

From the local minimality of $w_k$ we deduce that
\begin{align*}
\RDPMF_{n_k}(\beta,g_k,(a_k,b_k+\delta(n_k)),w_k) &\leq \RDPMF_{n_k}(\beta,g_k,(a_k,b_k+\delta(n_k)),v_k) \\[1ex]
&= \frac{\delta(n_k)}{\omega(n_k)^2}\left[\log\left(1+\frac{w_k(a_k)^2}{\delta(n_k) ^2}\right)+\log\left(1+\frac{w_k(b_k)^2}{\delta(n_k) ^2}\right)\right] \\[0.5ex]
& \quad + \beta\int_{a_{k}} ^{a_{k}+\delta(n_k)}(w_k-g_k)^2\,dx
+\beta\int_{a_{k}+\delta(n_k)} ^{b_{k}} g_k^2 \,dx\\[0.5ex]
&\quad + \beta\int_{b_{k}} ^{b_{k}+\delta(n_k)}(w_k-g_k)^2\,dx\\[1ex]
&\leq\frac{1}{\log n_k}\left[\log\left(1+\frac{w_k(a_k)^2}{\delta(n_k)^2}\right)+\log\left(1+\frac{w_k(b_k)^2}{\delta(n_k) ^2}\right)\right] \\[0.5ex]
&\quad + \beta \delta(n_k)\left(2w_k(a_k)^2+2 M_L^2\right) + \beta M_L^2 (2L+4)\\[0.5ex]
&\quad + \beta \delta(n_k)\left(2w_k(b_k)^2+2 M_L^2\right).
\end{align*}

Hence by (\ref{eq:est-w(ak)_w(bk)}) we obtain that
\begin{eqnarray*}
\RDPMF_{n_k}(\beta,g_k,(a_k,b_k),w_k) &\leq& \frac{2}{\log n_k}\log\left(1+\frac{2C}{\beta \delta(n_k)^2 \omega(n_k)}+\frac{4M_L ^2}{\delta(n_k) ^2}\right)\\
&&\mbox{} + \frac{4C\delta(n_k)}{\omega(n_k)} + \beta M_L^2 \left(12 \delta(n_k) + 2L+4\right),
\end{eqnarray*}
which implies that
\begin{equation}\label{eq:w_k_bounded_energy}
\sup_{k\in \N}\RDPMF_{n_k}(\beta,g_k,(a_k,b_k),w_k):=\Gamma <+\infty.\end{equation}

Now we can repeat the argument used in (\ref{argument_ak_bk}) starting with $c_k\in \delta(n_k) \mZ_{\delta(n_k)}(-L-1,L)$ and $d_k\in \delta(n_k) \mZ_{\delta(n_k)}(L,L+1)$ in place of $a_k$ and $b_k$, and the uniform estimate (\ref{eq:w_k_bounded_energy}) in place of 
(\ref{hp:RDPMF<C/delta}). We obtain that
$$w_k(c_k)^2 + w_k(d_k)^2 \leq \frac{2\Gamma}{\beta} + 4M_L ^2,$$
hence $w_k(c_k)$ and $w_k(d_k)$ are uniformly bounded.

Now we observe that the local minimality of $w_k$ implies that
\begin{equation}\label{eq:w_k_bounded}
|w_k(x)|\leq \max\{|w_k(c_k)|,|w_k(d_k)|,\|g_k\|_{L^\infty((c_k,d_k))}\}=:T_k \qquad \forall x \in [c_k,d_k+\delta(n_k))\quad \forall k\in\N.\end{equation}

Indeed, the function
$$w_k ^{T_k}(x):=\min\{T_k,\max\{w_k(x),-T_k\}\},$$
has the same boundary values of $w_k$ in $(c_k,d_k+\delta(n_k))$, and
$$\RDPMF_{n_k}(\beta,g_k,(c_k,d_k+\delta(n_k)),w_k ^{T_k}) \leq \RDPMF_{n_k}(\beta,g_k,(c_k,d_k+\delta(n_k)),w_k),$$
because both the term with the logarithm and the fidelity term decrease with the truncation. Since the inequality is strict if $w_k ^{T_k}\neq w_k$, by the local minimality of $w_k$ we deduce that necessarily $w_k ^{T_k}=w_k$, hence (\ref{eq:w_k_bounded}) holds true.

Combining (\ref{eq:w_k_bounded_energy}) and (\ref{eq:w_k_bounded}) we obtain (\ref{claim:compactness_w_k}), because the sequence $\{T_k\}\subset (0,+\infty)$ is bounded.

Therefore, we can apply statement~(1) in Theorem~\ref{thm:gamma-convergence} and we deduce that there exists a function $w_\infty \in S((-L,L))$ and a subsequence (not relabeled) such that $w_{k}\to w_\infty$ in $L^2 ((-L,L))$.

Since $L>0$ is arbitrary, with a diagonal argument we can extract a further subsequence in such a way that $w_{k}\to w_\infty$ in $L^2 _{loc}(\R)$ and almost everywhere, where now $w_\infty \in S_{loc}(\R)$.

Now we claim that $w_\infty$ is an entire local minimizer for $\JF(4/3,\beta,Mx,\cdot,\cdot)$ and that the convergence is actually locally strict in $BV_{loc}(\R)$.

To prove this, let $(a,b)\subset\R$ be an interval whose endpoints are not jump points of $w_\infty$ and such that $w_{k}(a)\to w_\infty(a)$ and $w_{k}(b)\to w_\infty(b)$ as $k\to +\infty$. We recall that jump points are at most countably many, and the pointwise convergence holds almost everywhere.

Let $v \in S((a,b))$ be any nonconstant step function with $v(a)=w_\infty(a)$ and $v(b)=w_\infty(b)$.

By statement (4) of Theorem~\ref{thm:gamma-convergence} for every $k\in\N$ we can find a function $v_k\in PC_{\delta(n_{k})}(a,b)$ with the same boundary values of $w_{k}$, namely $v_k(a)=w_{k}(a)$ and $v_k(b)=w_{k}(b)$, and such that $v_k \to v$ in $L^2((a,b))$ and
$$\lim_{k\to +\infty} \RDPM_{n_k}((a,b),v_k)=\frac{4}{3}\J((a,b),v).$$

From Theorem~\ref{thm:gamma-convergence}, the local minimality of $w_{k}$ and assumption (ii) we deduce that
\begin{eqnarray}
\JF(4/3,\beta,Mx,(a,b),w_\infty) &\leq& \liminf_{k\to +\infty} \RDPMF_{n_k}(\beta,g_k,(a,b),w_{k})\nonumber\\
&\leq& \limsup_{k\to +\infty} \RDPMF_{n_k}(\beta,g_k,(a,b),w_{k})\nonumber\\
&\leq& \lim_{k\to +\infty} \RDPMF_{n_k}(\beta,g_k,(a,b),v_{k})\nonumber\\
&=& \JF(4/3,\beta,Mx,(a,b),v).\label{est:w_infty_minimizer}
\end{eqnarray}

If $M\neq 0$, by Lemma~\ref{lemma:prop_minimizers} we know that constant functions can not be local minimizers in any interval $(a,b)$ with $b-a>2L_0$. Hence the fact that (\ref{est:w_infty_minimizer}) holds for every nonconstant $v$ is enough to deduce that $w_\infty$ is a local minimizer for $\JF(4,\beta,Mx,(a,b),w_\infty)$ in $(a,b)$, as soon as $b-a$ is large enough. Since the intervals $(a,b)$ for which this argument works invade the real line, we have that $w_\infty$ is also an entire local minimizer.

At this point, thanks to Proposition~\ref{prop:local_minimizers} we know that $w_\infty$ is a staircase, and in particular it is not constant on large intervals. So we can apply (\ref{est:w_infty_minimizer}) with $v=w_\infty$ and we obtain that $w_{k}$ is a recovery sequence for $w_\infty$ in $(a,b)$. By statement~(5) in Theorem~\ref{thm:gamma-convergence} this implies that $w_{k}\auto w_\infty$ in $BV((a,b))$. As before, since this holds for a family of intervals invading the real line, the convergence is actually locally strict in $BV_{loc}(\R)$.

Now we consider the case $M=0$, so we have to prove that $w_\infty\equiv 0$.

Let us consider again an interval $(a,b)$ whose endpoints are not jump points of $w_\infty$ and such that $w_{k}(a)\to w_\infty(a)$ and $w_{k}(b)\to w_\infty(b)$ as $k\to +\infty$.

If $w_\infty (a)=w_\infty(b)=0$, by statement~(4) of Theorem~\ref{thm:gamma-convergence} for every $k\in\N$ we can find a function $v_k\in PC_{\delta(n_k)}(a,b)$ with the same boundary values of $w_{k}$, namely $v_k(a)=w_{k}(a)$ and $v_k(b)=w_{k}(b)$, and such that $v_k \to 0$ in $L^2((a,b))$ and
$$\limsup_{k\to +\infty} \RDPM_{n_k}((a,b),v_k)\leq \frac{4}{3}.$$

Hence we have that 
\begin{eqnarray*}
\JF(4/3,\beta,0,(a,b),w_\infty) &\leq& \liminf_{k\to +\infty} \RDPMF_{n_k}(\beta,g_k,(a,b),w_{k})\\
&\leq& \limsup_{k\to +\infty} \RDPMF_{n_k}(\beta,g_k,(a,b),w_{k})\\
&\leq& \limsup_{k\to +\infty} \RDPMF_{n_k}(\beta,g_k,(a,b),v_{k})\\
&\leq& \frac{4}{3}.
\end{eqnarray*}

This implies that $w_\infty$ has at most one jump point in $(a,b)$, but since we assumed that $w_\infty (a)=w_\infty(b)=0$, the only possibility is that $w_\infty \equiv 0$ in $(a,b)$. Moreover, from statement~(6) of Theorem~\ref{thm:gamma-convergence} we also get strict convergence on every subinterval $(c,d)$.

Let us assume now that either $w_\infty(a)\neq 0$ or $w_\infty(b)\neq 0$. Then we can apply (\ref{est:w_infty_minimizer}) with the nonconstant function $v_\tau \in S((a,b))$ defined as
$$v(x)=\begin{cases}
w_\infty(a) &\mbox{if }x \in (a,a+\tau),\\
0 &\mbox{if }x \in [a+\tau,b-\tau),\\
w_\infty(b) &\mbox{if }x \in [b-\tau,b),
\end{cases}$$
where $\tau \in (0,(b-a)/2)$. We obtain that
$$\JF(4/3,\beta,0,(a,b),w_\infty) \leq \JF(4/3,\beta,0,(a,b),v)\leq \frac{8}{3} + \beta w_\infty(a)^2 \tau + \beta w_\infty (b)^2 \tau.$$

Letting $\tau \to 0$, we get that
$$\frac{4}{3}\J((a,b),w_\infty)+\beta \int_a ^b w_\infty(x)^2\,dx=\JF(4/3,\beta,0,(a,b),w_\infty) \leq \frac{8}{3}.$$

Since we assumed that $w_\infty$ is not identically zero, this implies that either $w_\infty$ is a constant function, different from zero, or it has exactly one jump point in $(a,b)$.

In the end, we have proved that whichever the boundary values of $w_\infty$ are, $w_\infty$ is either constant or has exactly one jump. Moreover, the value $\JF(4/3,\beta,0,(a,b),w_\infty)$ is bounded by $8/3$ independently of the interval $(a,b)$, provided the endpoints are chosen outside a negligible set. But this is impossible if $w_\infty$ is different from zero on a half-line, hence the only possibility is that $w_\infty \equiv 0$ on $\R$.

At this point we recall that we have proved that if $w_\infty \equiv 0$ on an interval $(a,b)$ then we have strict convergence of $w_{k}$ to $w_\infty$ on every interval $(c,d)$ with $a<c<d<b$. Since we now know that $w_\infty$ vanishes on $\R$, we actually have that $w_{k} \auto w_\infty$ locally strictly in $BV_{loc}(\R)$.

The last part of the statement follows from statement (5) of Theorem~\ref{thm:gamma-convergence} when $M\neq 0$ and directly from the strict convergence when $M=0$.
\end{proof}

\setcounter{equation}{0}
\section{Proofs of the main results}

\subsection{Proof of Theorem~\ref{thm:varifold}}

\subsubsection{Proof of statement (1)}
Since we know that $u_n\to f$ in $L^2(0,1)$, it is enough to show that
$$\limsup_{n\to +\infty} |Du_n|((0,1))\leq |Df|((0,1)).$$

We prove that actually
\begin{equation}\label{eq:TV_un<TV_f}
|Du_n|((0,1))\leq |Df|((0,1)) \qquad \forall n\in \N.
\end{equation}

The idea is rather simple, but a complete formalization would be quite involved, so we try to explain it in words as much as possible. First of all, we observe that (\ref{eq:TV_un<TV_f}) is trivial if $u_n$ is constant. Otherwise, since $u_n$ is a discrete function, we can compute its total variation as a finite sum of increments with alternating sign, namely we can find finitely many integer numbers $0\leq k_1<\dots <k_m\leq n-1$ such that
$$|Du_n|((0,1))=\left|\sum_{i=1} ^{m-1} (-1)^i\left[u_n\left(\frac{k_{i+1}}{n}\right)- u_n\left(\frac{k_{i}}{n}\right)\right]\right|,$$
and that the points $k_i/n$ are alternatively local maxima or local minima of $u_n$ in the discrete sense, namely $u_n(k_i/n)$ is larger (o smaller) than or equal to the two values $u_n((k_i\pm 1)/n)$.

Now, we observe that if $k/n$ is a strict local maximum point (in the discrete sense) for $u_n$, we can find a point $x\in I_{1/n,k}=[k/n,(k+1)/n)$ such that $f(x)\geq u_n(k/n)$, otherwise we could reduce the value of $\DPMF_n(\beta,f,(0,1) ,u_n)$ by lowering the value of $u_n(k/n)$, thus contradicting the minimality of $u_n$. A symmetric statement holds for strict local minima, so we end up with finitely many points $0\leq x_1< \dots < x_m \leq 1$ that we can use to estimate the total variation of $f$ from below, so we obtain (\ref{eq:TV_un<TV_f}).

Unfortunately, there are some complications, because the points $k_i/n$ might not be strict local maxima or minima, so we need to work with the ``maximal intervals of local maximality'' or minimality of $u_n$ in order to find the points $x_i$.

To be more precise, for every $i\in\{1,\dots,m\}$ let $[a_i,b_i)$ be the maximal interval containing $k_i/n$ on which $u_n$ is constant. Of course we have that $a_i<b_i$ for every $i\in\{1,\dots,m\}$ and $b_i\leq a_{i+1}$ for every $i\in\{1,\dots,m-1\}$.

%Moreover, if we set by definition $b_0:=0$ and $a_{m+1}:=1$, then for every odd index $i$, it holds that
%$$u_n(\bar{k}_i)=\max\{ u_n(x): x\in (b_{i-1},a_{i+1})\},$$
%while for every even index $i$ it holds that
%$$u_n(\bar{k}_i)=\min\{u_n(x): x\in (b_{i-1},a_{i+1})\},$$
%because of the maximality property of $(\bar{k}_1,\dots,\bar{k}_m$).

%Now we claim that for every odd index $i$ it holds that

Now, if $[a_i,b_i)$ is an interval of local maximality of $u_n$, namely if
$$u_n(k_i/n)>u_n(a_i-1/n)\qquad \mbox{and} \qquad u_n(k_i/n)>u_n(b_i)$$
(or only one of the two if $[a_i,b_i)$ touches the boundary of $(0,1)$), then
\begin{equation}\label{eq:sup_f_odd}
\sup\{f(x):x\in (a_i,b_i)\}\geq u_n(k_i/n),
\end{equation}
namely for every $\ep>0$ there exists a point $x_i \in (a_i,b_i)$ such that
$$f(x_i)\geq u_n(k_i/n)-\frac{\ep}{2(m-1)}$$

Indeed, if we assume by contradiction that (\ref{eq:sup_f_odd}) does not hold, we can consider the function $v_{n,\tau}\in PC_n(0,1)$ that coincides with $u_n$ outside $[a_i,b_i)$ and is equal to $u_{n}(k_i/n)-\tau$ in $[a_i,b_i)$, and we deduce that
$$\DPMF_n(\beta,f,(0,1),v_{n,\tau}) <\DPMF_n(\beta,f,(0,1),u_n),$$
at least when $\tau>0$ is sufficiently small, because both the terms in the functional decrease when passing from $u_n$ to $v_{n,\tau}$. This contradicts the minimality of $u_n$.

Analogously, if $[a_i,b_i)$ is an interval of local minimality, then for every $\ep>0$ there exists a point $x_i\in (a_i,b_i)$ such that
$$f(x_i)\leq u_n(k_i/n)+\frac{\ep}{2(m-1)}$$

Since the intervals $[a_i,b_i)$ are alternatively intervals of local maximality and minimality we obtain that
\begin{eqnarray*}
|Du_n|((0,1))&=&\left|\sum_{i=1} ^{m-1} (-1)^i\left[u_n\left(\frac{k_{i+1}}{n}\right)- u_n\left(\frac{k_{i}}{n}\right)\right]\right|\\
&\leq&\left|\sum_{i=1} ^{m-1}  (-1)^i [f(x_{i+1}) - f(x_i)] \right| + \ep\\
&\leq& |Df|((0,1)) +\ep.
\end{eqnarray*}

Letting $\ep\to 0$ we obtain (\ref{eq:TV_un<TV_f}).

\color[rgb]{0,0,0}

\subsubsection{Proof of statement (2)}
First of all, we observe that the strict convergence of $\{u_n\}$ to $f$ implies that also $\{\widehat{u}_n\}$ converges strictly to $f$, hence by Remark~\ref{rmk:strict} we deduce that
\begin{equation}\label{th:conv-measure}
D_+\widehat{u}_n=\widehat{u}_n ' \Leb^1 \mres V_n ^+ \stackrel{*}{\weakto} D_+f
\qquad \mbox{and} \qquad
D_-\widehat{u}_n=|\widehat{u}_n '| \Leb^1 \mres V_n ^- \stackrel{*}{\weakto} D_-f,
\end{equation}
where $\Leb^1$ denotes the Lebesgue measure,
\begin{equation*}
V_n ^+:=\{x\in (0,1):\widehat{u}_n '(x)>0\} \qquad\mbox{and} \qquad V_n ^-:=\{x\in (0,1):\widehat{u}_n'(x)<0\}.
\end{equation*}

Moreover, the strict convergence implies that $\{\widehat{u}_n\}$ is uniformly bounded in $C^0([0,1])$, so we can set
$$T:=\sup_{n\in\N} \|\widehat{u}_n\|_\infty <+\infty,$$
and we observe that $\|f\|_\infty\leq T$.

Since $\phi$ is continuous, there exists a constant $M_0>0$ such that
\begin{equation}
|\phi(x,s,\arctan p)|\leq M_{0}
\qquad
\forall(x,s,p)\in[0,1]\times [-T,T]\times\R
\label{est:bound-phi}
\end{equation}

Now for every $a\in(0,1)$ we define the three sets
\begin{gather*}
\Omega_{a}:=\left\{((x,y),(s,t),p)\in[0,1]^2 \times [-T,T]^2 \times\R:|y-x|\leq a,\ |s-t|\leq a,\ |p|\leq a\right\},
\\[1ex]
\Omega_{a}^{+}:=\left\{((x,y),(s,t),p)\in[0,1]^2 \times [-T,T]^2 \times\R:|y-x|\leq a,\ |s-t|\leq a,\ p\geq 1/a\right\},
\\[1ex]
\Omega_{a}^{-}:=\left\{((x,y),(s,t),p)\in[0,1]^2 \times [-T,T]^2 \times\R:|y-x|\leq a,\ |s-t|\leq a,\ p\leq-1/a\right\},
\end{gather*}
and the corresponding three constants
\begin{gather*}
\Gamma_{a}:=\max\left\{|\phi(y,t,\arctan p)-\phi(x,s,0)|:((x,y),(s,t),p)\in \Omega_{a}\right\},
\\[1ex]
\Gamma_{a}^{+}:=\max\left\{|\phi(y,t,\arctan p)-\phi(x,s,\pi/2)|:((x,y),(s,t),p)\in \Omega_{a}^{+}\right\},
\\[1ex]
\Gamma_{a}^{-}:=\max\left\{|\phi(y,t,\arctan p)-\phi(x,s,-\pi/2)|:((x,y),(s,t),p)\in \Omega_{a}^{-}\right\}.
\end{gather*}

We observe that, due to the uniform continuity of $\phi$ in bounded sets, these constants satisfy
\begin{equation*}
\lim_{a\to 0^{+}}\Gamma_{a}=\lim_{a\to 0^{+}}\Gamma_{a}^{+}=\lim_{a\to 0^{+}}\Gamma_{a}^{-}=0.
%\label{lim-gamma-vari}
\end{equation*}

Now for every $a \in (0,1)$ we consider the set $S_a:=\{x \in S_f: |Df|(\{x\})>a/2\}$, and we fix an open neighborhood $\Sigma_a$ of $S_a$ with the following properties:
\begin{itemize}
\item if $S_a$ consists of $m_a$ points, then $\Sigma_a$ is the union of $m_a$ disjoint open intervals $I_1,\dots,I_{m_a}$ such that $\overline{I_i} \cap S_a$ is a single point and $|Df|(\partial I_i)=0$ for every $i\in \{1,\dots,m_a\}$,

\item the following estimates hold
$$\Leb^1(\Sigma_a)\leq a,\qquad |Df|(\Sigma_a)\leq |Df|(S_a)+a.$$
\end{itemize}

From the definition of $S_a$ we also deduce that
$$\ep_a:=\sum_{x \in S_f\setminus S_a}|Df|(\{x\})=|Df|(S_f \setminus S_a)$$
tends to $0$ as $a\to 0^+$.

Now we observe that the strict convergence implies that for every $a\in (0,1)$ there exists a positive integer $n_a\in \N$ such that
\begin{equation}\label{est:Linfty_out_Sigma}
\|\widehat{u}_n - f\|_{L^\infty([0,1]\setminus \Sigma_a)}\leq a \qquad \forall n\geq n_a,\end{equation}
and of course we can choose $n_a$ so that $n_a\to +\infty$ as $a\to 0^+$.

Indeed, if this were false, we could find a sequence $\{x_k\} \subset [0,1]\setminus \Sigma_a$ and a diverging sequence of positive integers $\{n_k\}$ such that
\begin{equation}\label{eq:unk-f>a}
|\widehat{u}_{n_k}(x_k)-f(x_k) |>a \qquad \forall k\in \N.
\end{equation}

Up to the extraction of a subsequence, we can assume that $x_k\to x_\infty \in [0,1]\setminus \Sigma_a$, hence the strict convergence yields
$$\liminf_{x\to x_\infty} f(x) \leq \liminf_{k\to +\infty} \widehat{u}_{n_k}(x_k)\leq \limsup_{k\to +\infty} \widehat{u}_{n_k}(x_k)\leq \limsup_{x\to x_\infty} f(x).$$

Since $x_\infty \notin S_a$, the left-hand side and the right-hand side differ at most by $a/2$, and this contradicts (\ref{eq:unk-f>a}).

Finally, for every $n\in\N$ and every $a\in(0,1)$, we write the interval $[0,1]$ as the disjoint union of the six sets
\begin{gather}
H_{a,n}:=\left\{x\in[0,1]:|\widehat{u}_n'(x)|\leq a\right\},
\label{defn:H-a-n}
\\[1ex]
V_{a,n}^{+}:=\left\{x\in[0,1]\setminus \Sigma_a:\widehat{u}_n'(x)\geq 1/a\right\},
\qquad
V_{a,n}^{-}:=\left\{x\in[0,1]\setminus \Sigma_a:\widehat{u}_n'(x)\leq -1/a\right\},
\label{defn:V-a-n}
\\[1ex]
J_{a,n}^{+}:=\left\{x\in \Sigma_a:\widehat{u}_n'(x)\geq 1/a\right\},
\qquad
J_{a,n}^{-}:=\left\{x\in\Sigma_a:\widehat{u}_n'(x)\leq -1/a\right\},
\label{defn:J-a-n}
\\[1ex]
M_{a,n}:=\left\{x\in[0,1]:a<|\widehat{u}_n'(x)|<1/a\right\},
\label{defn:M-a-n}
\end{gather}
and accordingly we write
\begin{equation*}
\int_{0}^{1}
\phi\left(x,\widehat{u}_{n}(x),\arctan(\widehat{u}_n'(x))\strut\right)\sqrt{1+\widehat{u}_n'(x)^{2}}\,dx=
\mathcal{I}_{a,n}^{H}+\mathcal{I}_{a,n}^{V^+}+\mathcal{I}_{a,n}^{V^-}+\mathcal{I}_{a,n}^{J^+}+\mathcal{I}_{a,n}^{J^-}+\mathcal{I}_{a,n}^{M},
\end{equation*}
where the six terms in the right-hand side are the integrals over the six sets defined above. We observe that
\begin{eqnarray*}
\DPMF_{n}(\beta,f,(0,1),u_{n}) & \geq & 
\int_{0}^{1}\log\left(1+D^{1/n}u_n(x)^{2}\right)\,dx
\\[0.5ex]
&=&\int_{0}^{1}\log\left(1+\widehat{u}_n '(x)^{2}\right)\,dx
\\[0.5ex]
& \geq &
\log\left(1+a^{-2}\right)\left(\Leb^1(V_{a,n}^{+})+\Leb^1(V_{a,n}^{-})+\Leb^1(J_{a,n}^{+})+\Leb^1(J_{a,n}^{-})\right)\\[0.5ex]
& &\mbox{}+\log\left(1+a^{2}\right)\Leb^1(M_{a,n}),
\end{eqnarray*}
and, since (\ref{m_to_0}) implies that the left-hand side tends to~0, we deduce that
\begin{equation*}
\lim_{n\to +\infty}\Leb^1(V_{a,n}^{+})+\Leb^1(V_{a,n}^{-})+\Leb^1(J_{a,n}^{+})+\Leb^1(J_{a,n}^{-})+\Leb^1(M_{a,n})=0
\qquad
\forall a\in(0,1),
\end{equation*}
and as a consequence
\begin{equation*}
\lim_{n\to +\infty}\Leb^1(H_{a,n}^{+})=1
\qquad
\forall a\in(0,1).
\end{equation*}

We claim that for every fixed $a\in(0,1)$ it turns out that
\begin{eqnarray}
\limsup_{n\to +\infty}\left|
\mathcal{I}_{a,n}^{H}-\int_{0}^{1}\phi(x,f(x),0)\,dx\right|
&\leq& M_{0}\left(\sqrt{1+a^{2}}-1+2a\right)+\Gamma_{a},
\label{th:limsup-H}
\\[1ex]
\lim_{n\to +\infty}I_{a,n}^{M}&=&0,
\label{th:limsup-M}
\\[1ex]
\limsup_{n\to +\infty}\left|
\mathcal{I}_{a,n}^{V^+}-\int_0 ^1 \phi(x,f(x),\pi/2)\,d\widetilde{D}_+f(x)\right|
&\leq& 3 \Gamma_{a}^{+} \cdot |Df|((0,1)) +M_{0}(2a+\ep_a),\quad\quad
\label{th:limsup-V+}
\\[1ex]
\limsup_{n\to +\infty}\left|
\mathcal{I}_{a,n}^{V^-}-\int_0 ^1 \phi(x,f(x),-\pi/2)\,d\widetilde{D}_-f(x)\right|
&\leq& 3 \Gamma_{a}^{-} \cdot |Df|((0,1)) +M_{0}(2a+\ep_a),\quad\quad
\label{th:limsup-V-}
\\[1ex]
\limsup_{n\to +\infty}\left|
\mathcal{I}_{a,n}^{J^+}-\sum_{x\in S_{f}^+}\int_{\mathcal{J}_f(x)} \phi(x,s,\pi/2)\,ds\right|
&\leq& 2\Gamma_a ^+ \cdot |Df|((0,1)) + M_0 (3a+\ep_a),
\label{th:limsup-J+}
\\[1ex]
\limsup_{n\to +\infty}\left|
\mathcal{I}_{a,n}^{J^-}-\sum_{x\in S_{f}^-}\int_{\mathcal{J}_f(x)} \phi(x,s,-\pi/2)\,ds\right|
&\leq& 2\Gamma_a ^- \cdot |Df|((0,1)) + M_0 (3a+\ep_a).
\label{th:limsup-J-}
\end{eqnarray}

If we prove these claims, then we let $a\to 0^{+}$ and we obtain (\ref{th:varifold}). 

In words, this means that the integral in the left-hand side of (\ref{th:varifold}) splits into the six integrals over the regions (\ref{defn:H-a-n}), (\ref{defn:V-a-n}), (\ref{defn:J-a-n}), (\ref{defn:M-a-n}), which behave as follows. 
\begin{itemize}

\item The integral over the ``intermediate'' region $M_{a,n}$ disappears in the limit.

\item  The integral over the ``horizontal'' region $H_{a,n}$ tends to the first integral in the right hand side of (\ref{th:varifold}), in which the ``tangent component'' is horizontal.

\item  The integrals over the two ``vertical'' regions far from the jump points $V^{+}_{a,n}$ and $V^{-}_{a,n}$ tend to the two integrals with respect to $\widetilde{D}_+f$ and $\widetilde{D}_-f$ in the right hand side of (\ref{th:varifold}). In this two integrals the ``tangent component'' is vertical.

\item  The integrals over the two ``vertical'' regions around the jump points $J^{+}_{a,n}$ and $J^{-}_{a,n}$ tend to the two sums in the right hand side of (\ref{th:varifold}). In this two sums the ``tangent component'' is also vertical.

\end{itemize}  

\paragraph{\textmd{\textit{Estimate in the intermediate regime}}}

From (\ref{est:bound-phi}) we know that
\begin{equation*}
\left|\phi\left(x,\widehat{u}_{n}(x),\arctan(\widehat{u}_n '(x))\right)\right|\sqrt{1+\widehat{u}_n '(x)^{2}}\leq
M_{0}\sqrt{1+\frac{1}{a^{2}}}
\qquad
\forall x\in M_{a,n},
\end{equation*}
and therefore
\begin{equation*}
|\mathcal{I}_{a,n}^{M}|\leq M_{0}\sqrt{1+\frac{1}{a^{2}}}\cdot \Leb^1(M_{a,n}).
\end{equation*}

Since $\Leb^1(M_{a,n})\to 0$ as $n\to +\infty$, this proves (\ref{th:limsup-M}).

\paragraph{\textmd{\textit{Estimate in the horizontal regime}}}

In order to prove (\ref{th:limsup-H}), we observe that
\begin{eqnarray*}
\mathcal{I}_{a,n}^{H}-\int_{0}^{1}\phi(x,f(x),0)\,dx &=& \int_{H_{a,n}}\phi\left(x,\widehat{u}_{n}(x),\arctan(\widehat{u}_n'(x))\strut\right)\left(\sqrt{1+\widehat{u}_n'(x)^{2}}-1\right)\,dx
\\
&& +\int_{H_{a,n}\setminus \Sigma_a}\left\{\phi\left(x,\widehat{u}_{n}(x),\arctan(\widehat{u}_n'(x))\strut\right)-\phi(x,f(x),0)\right\}\,dx,
\\
&& +\int_{H_{a,n}\cap \Sigma_a}\left\{\phi\left(x,\widehat{u}_{n}(x),\arctan(\widehat{u}_n'(x))\strut\right)-\phi(x,f(x),0)\right\}\,dx,
\\
&& +\int_{H_{a,n}}\phi(x,f(x),0)\,dx-
\int_{0}^{1}\phi(x,f(x),0)\,dx.
\end{eqnarray*}

The absolute value of the first integral in the right-hand side is less than or equal to $M_{0}\left(\sqrt{1+a^{2}}-1\right)$. The absolute value of the second line is less than or equal to $\Gamma_{a}$ provided that
\begin{equation}
|\widehat{u}_{n}(x)-f(x)|\leq a
\qquad
\forall x\in[0,1]\setminus \Sigma_a,
\label{est:uep-f}
\end{equation}
and this happens whenever $n\geq n_a$ thanks to (\ref{est:Linfty_out_Sigma}). The absolute value of the third line is less than $2M_0a$, because $\Leb^1(\Sigma_a)\leq a$. The fourth line tends to~0 because $|H_{a,n}|\to 1$ as $n\to +\infty$. This is enough to establish (\ref{th:limsup-H}).

\paragraph{\textmd{\textit{Estimate in the vertical regime far from jump points}}}

In order to prove (\ref{th:limsup-V+}), we observe that
\begin{eqnarray*}
\lefteqn{\hspace{-4em}
\mathcal{I}_{a,n}^{V^+}-\int_0 ^1 \phi(x,f(x),\pi/2)\,d\widetilde{D}_+f(x) =
}
\\  
\qquad\qquad & = & \int_{V_{a,n}^{+}}\phi\left(x,\widehat{u}_{n}(x),\arctan(\widehat{u}_{n}'(x))\strut\right)\left(\sqrt{1+\widehat{u}_{n}'(x)^{2}}-\widehat{u}_{n}'(x)\right)\,dx
\\[0.5ex]
& & +\int_{V_{a,n}^{+}}\left\{\phi\left(x,\widehat{u}_{n}(x),\arctan(\widehat{u}_{n}'(x))\strut\right)-\phi(x,\widehat{u}_{n_a}(x),\pi/2)\right\}\widehat{u}_{n}'(x)\,dx,
\\[0.5ex]
& & +\int_{V_{a,n}^{+}}\phi(x,\widehat{u}_{n_a}(x),\pi/2)\widehat{u}_{n}'(x)\,dx
-\int_{V_{n}^{+}\setminus \Sigma_a}\phi(x,\widehat{u}_{n_a}(x),\pi/2)\widehat{u}_{n}'(x)\,dx
\\[0.5ex]
& & +\int_{V_{n}^{+}\setminus \Sigma_a}\phi(x,\widehat{u}_{n_a}(x),\pi/2)\widehat{u}_{n}'(x)\,dx
-\int_{[0,1]\setminus\Sigma_a}\phi(x,\widehat{u}_{n_a}(x),\pi/2) \,dD_+f(x)
\\[0.5ex]
& & +\int_{[0,1]\setminus\Sigma_a}\phi(x,\widehat{u}_{n_a}(x),\pi/2)\,d(D_+f-\widetilde{D}_+f)(x)
\\[0.5ex]
& & +\int_{[0,1]\setminus\Sigma_a}\phi(x,\widehat{u}_{n_a}(x),\pi/2)\,d\widetilde{D}_+f(x)
-\int_{[0,1]\setminus\Sigma_a}\phi(x,f(x),\pi/2)\,d\widetilde{D}_+f(x)
\\[0.5ex]
& & +\int_{[0,1]\setminus\Sigma_a}\phi(x,f(x),\pi/2)\,d\widetilde{D}_+f(x)
-\int_0 ^1 \phi(x,f(x),\pi/2)\,d\widetilde{D}_+f(x)
\\[0.5ex]
& =: & L_{1}^V+L_{2}^V+L_{3}^V+L_{4}^V+L_{5}^V+L_{6}^V+L_{7}^V.
\end{eqnarray*}

Let us consider the seven lines separately. The first line can be estimated as
\begin{equation*}
|L_{1}^V|\leq M_{0}\max\left\{\sqrt{1+p^{2}}-p:p\geq 1/a\right\}\Leb^1(V_{a,n}^{+})\leq
M_{0}\cdot\frac{a}{2}\cdot \Leb^1(V_{a,n}^{+}),
\end{equation*}
and this tends to~0 when $n\to +\infty$. The second line can be estimated as
\begin{eqnarray*}
|L_{2}^V|&\leq& \int_{V_{a,n}^{+}}\left\{|\phi\left(x,\widehat{u}_{n}(x),\arctan(\widehat{u}_{n}'(x))\right)-\phi(x,f(x),\pi/2)|\right.\\
&&\qquad\ \left. \mbox{}+|\phi(x,f(x),\pi/2)-\phi(x,\widehat{u}_{n_a}(x),\pi/2)|\strut \right\}\widehat{u}_{n}'(x)\,dx,\\
&\leq& 2\Gamma_{a}^{+}\cdot\int_{0}^{1}|\widehat{u}_{n}'(x)|\,dx
\end{eqnarray*}
thanks to (\ref{est:uep-f}), provided $n\geq n_a$. Hence, from the strict convergence of $\{\widehat{u}_n\}$ we deduce that 
$$\limsup_{n\to +\infty}|L_{2}^V| \leq\Gamma_{2a}^{+}\cdot |Df|((0,1)).$$

For the third line we observe that $V_{a,n}^+$ is a subset of $V_n ^+ \setminus \Sigma_a$ and $(V_{n}^{+}\setminus \Sigma_a)\setminus V_{a,n}^{+}\subseteq H_{a,n}\cup M_{a,n}$, and therefore
\begin{eqnarray*}
|L_{3}^V| & \leq  &
\int_{H_{a,n}}|\phi(x,\widehat{u}_{n_a}(x),\pi/2)|\cdot|\widehat{u}_{n}'(x)|\,dx+
\int_{M_{a,n}}|\phi(x,\widehat{u}_{n_a}(x),\pi/2)|\cdot|\widehat{u}_{n}'(x)|\,dx
\\
& \leq &
M_{0}a+M_{0}\cdot\frac{1}{a}\cdot \Leb^1(M_{a,n}),
\end{eqnarray*}
and this tends to $M_0 a$ when $n\to +\infty$. 

For the fourth line, we observe that $L_{4}^V\to 0$ as $n\to +\infty$ because of (\ref{th:conv-measure}), the fact that $|Df|(\partial \Sigma_a)=0$ and the continuity of $\widehat{u}_{n_a}$.

For the fifth line, we recall that $D_+f-\widetilde{D}_+f=D_+ ^Jf$, hence we have that
$$|L_5 ^V|\leq M_0 \cdot |D_+ ^J f|((0,1)\setminus \Sigma_a)\leq M_0 \ep_a$$

For the sixth line, similarly to the second one, we have that
$$|L_6 ^V|\leq \Gamma_a ^+ \cdot \widetilde{D}_+f((0,1))\leq \Gamma_a ^+ \cdot |Df|((0,1)).$$

Finally, for the seventh line we observe that
$$|L_7 ^V|\leq M_0 \cdot \widetilde{D}_+ f(\Sigma_a) \leq M_0\cdot |Df|(\Sigma_a \setminus S_a) \leq M_0 a.$$

From the previous estimates we conclude that
$$\limsup_{n\to +\infty} |L_{1}^V+L_{2}^V+L_{3}^V+L_{4}^V+L_{5}^V+L_{6}^V+L_{7}^V|\leq 3 \Gamma_{a}^{+} \cdot |Df|((0,1)) +M_{0}(2a+\ep_a),$$
which proves (\ref{th:limsup-V+}). The proof of (\ref{th:limsup-V-}) is analogous.

\paragraph{\textmd{\textit{Estimate in the vertical regime around the jump points}}}

In order to prove (\ref{th:limsup-J+}), let us first split the set $S_a$ into two sets $S_a ^+$ and $S_a ^-$ depending on the sign of the jumps, and let us label the points in $S_a^+$ as follows
$$S_a ^+:=S_a\cap S_{f}^+=\{y_1,\dots, y_{m_a^+}\}.$$

Similarly, we split $\Sigma_a$ into two sets $\Sigma_a ^+$ and $\Sigma_a ^-$ each consisting of the intervals $I_i$ containing points in $S_a ^+$ or $S_a ^-$, and we relabel the intervals $I_i$ in such a way that $\Sigma_a ^+ = I_1 \cup\dots\cup I_{m_a ^+}$, with $\overline{I}_i\cap S_a^+=\{y_i\}$ for every $i\in\{1,\dots,m_a^+\}$.

Now we observe that
\begin{eqnarray*}
\lefteqn{\hspace{-4em}
\mathcal{I}_{a,n} ^{J^+} -\sum_{x\in S_f^+}\int_{\mathcal{J}_f(x)} \phi(x,s,\pi/2)\,ds=}
\\
\qquad\qquad & = & \int_{J_{a,n}^{+}}\phi\left(x,\widehat{u}_{n}(x),\arctan(\widehat{u}_{n}'(x))\strut\right)\left(\sqrt{1+\widehat{u}_{n}'(x)^{2}}-\widehat{u}_{n}'(x)\right)\,dx
\\[0.5ex]
&&+\int_{J_{a,n}^{+}}\left\{\phi\left(x,\widehat{u}_{n}(x),\arctan(\widehat{u}_{n}'(x))\strut\right)-\phi(x,\widehat{u}_{n}(x),\pi/2)\right\}\widehat{u}_n'(x)\,dx
\\[0.5ex]
&&+\int_{J_{a,n}^{+}} \phi(x,\widehat{u}_{n}(x),\pi/2)\widehat{u}_n'(x)\,dx
- \int_{V_n ^+\cap \Sigma_a} \phi(x,\widehat{u}_{n}(x),\pi/2)\widehat{u}_n'(x)\,dx
\\[0.5ex]
&&+\int_{V_n ^+\cap \Sigma_a} \phi(x,\widehat{u}_{n}(x),\pi/2)\widehat{u}_n'(x)\,dx
%- \int_{V_n ^+\cap \Sigma_a ^+} \phi(x,\widehat{u}_{n}(x),\pi/2)\widehat{u}_n'(x)\,dx
%\\[0.5ex]
%&&+\int_{V_n ^+\cap \Sigma_a ^+} \phi(x,\widehat{u}_{n}(x),\pi/2)\widehat{u}_n'(x)\,dx
- \int_{\Sigma_a ^+} \phi(x,\widehat{u}_{n}(x),\pi/2)\widehat{u}_n'(x)\,dx
\\[0.5ex]
&&+\int_{\Sigma_a ^+} \phi(x,\widehat{u}_{n}(x),\pi/2)\widehat{u}_n'(x)\,dx
-\sum_{x\in S_a ^+}\int_{\mathcal{J}_f(x)} \phi(x,s,\pi/2)\,ds
\\[0.5ex]
&&+\sum_{x\in S_a ^+}\int_{\mathcal{J}_f(x)} \phi(x,s,\pi/2)\,ds
-\sum_{x\in S_{f}^+}\int_{\mathcal{J}_f(x)} \phi(x,s,\pi/2)\,ds
\\[0.5ex]
&=:&L_1^J+L_2^J+L_3^J+L_4^J+L_5^J+L_6^J.
\end{eqnarray*}

Let us consider the six lines separately. The first three lines are similar to the first three lines in the previous paragraph, and can be estimated as follows
\begin{gather*}
|L_1^J|\leq M_0 \cdot \frac{a}{2}\cdot \Leb^1(J_{a,n}^+),\\
|L_2^J|\leq \Gamma_a ^+ \int_0 ^1 |\widehat{u}_n'(x)|\,dx,\\
|L_3^J|\leq M_0a + M_0\cdot \frac{1}{a}\cdot \Leb^1(M_{a,n}),
\end{gather*}
hence
$$\limsup_{n\to +\infty} |L_1^J + L_2^J + L_3^J|\leq M_0a + \Gamma_a ^+ \cdot |Df|((0,1)).$$

For the fourth line, we observe that
\begin{eqnarray*}
|L_4^J|&=&\left|\int_{V_n ^+\cap \Sigma_a ^-} \phi(x,\widehat{u}_{n}(x),\pi/2)\widehat{u}_n'(x)\,dx-\int_{V_n ^-\cap \Sigma_a ^+} \phi(x,\widehat{u}_{n}(x),\pi/2)\widehat{u}_n'(x)\,dx\right|\\
&\leq& M_0 \int_{V_n ^+\cap \Sigma_a ^-} |\widehat{u}_n'(x)|\,dx+ M_0\int_{V_n ^-\cap \Sigma_a ^+} |\widehat{u}_n'(x)|\,dx
\end{eqnarray*}

Then, from (\ref{th:conv-measure}) and the properties of $\Sigma_a$, we deduce that
\begin{eqnarray*}
\limsup_{n\to +\infty}|L_4^J|&\leq& M_0 \cdot (D_+f(\Sigma_a ^-)+D_-f(\Sigma_a ^+))\\
&=&M_0 \cdot (D_+f(\Sigma_a ^-\setminus S_a ^-)+D_-f(\Sigma_a ^+\setminus S_a ^+))\\
&\leq& M_0 \cdot |Df|(\Sigma_a \setminus S_a)\\
&\leq& M_0 a.
\end{eqnarray*}

%The fifth line is analogous to the fourth one, because
%$$|L_5^J|=\left|\int_{V_n ^-\cap \Sigma_a ^+} \phi(x,\widehat{u}_{n}(x),\pi/2)\widehat{u}_n'(x)\,dx\right|\leq M_0 \int_{V_n ^-\cap \Sigma_a ^+} |\widehat{u}_n'(x)|\,dx,$$
%and therefore, as before, we obtain that
%$$\limsup_{n\to +\infty}|L_5^J|\leq M_0\cdot D_-f(\Sigma_a^+)\leq M_0 a.$$

Now let us estimate the sixth line before the fifth one. To this end, it is enough to observe that
$$|L_6 ^J|= \left|\sum_{x \in S_{f}^+\setminus S_a ^+} \int_{\mathcal{J}_f(x)} \phi(x,s,\pi/2)\,ds\right|\leq \sum_{x \in S_{f}^+\setminus S_a ^+} M_0 \cdot |Df|(\{x\}) \leq M_0 \ep_a.$$

Now we have to estimate the fifth line, which requires more work than the others. We observe that
$$L_5^J=\sum_{i=1} ^{m_a ^+} \int_{I_i} \phi(x,\widehat{u}_{n}(x),\pi/2)\widehat{u}_n'(x)\,dx -\int_{\mathcal{J}_f(y_i)} \phi(y_i ,s,\pi/2)\,ds,$$
and that for every $i\in \{1,\dots,m_a ^+\}$ it holds that
$$\left|\int_{I_i} \phi(x,\widehat{u}_{n}(x),\pi/2)\widehat{u}_n'(x)\,dx- \int_{I_i} \phi(y_i,\widehat{u}_{n}(x),\pi/2)\widehat{u}_n'(x)\,dx\right| \leq \Gamma_a ^+ \int_{I_i} |\widehat{u}_n '(x)| \,dx.$$

Therefore we have that
\begin{equation}\label{est:L5J}
|L_5^J|\leq \Gamma_a ^+ \int_{\Sigma_a ^+} |\widehat{u}_n '(x)| \,dx +\mbox{} \sum_{i=1} ^{m_a ^+} \left| \int_{I_i} \phi(y_i,\widehat{u}_{n}(x),\pi/2)\widehat{u}_n'(x)\,dx - \int_{\mathcal{J}_f(y_i)} \phi(y_i ,s,\pi/2)\,ds\right|.
\end{equation}

For $i \in \{1,\dots,m_a^+\}$ we set $I_i =(\alpha_i,\beta_i)$ and we observe that $f(y_i ^-)<f(y_i^+)$, because $Df(\{y_i\})>0$, and hence we have that $\mathcal{J}_f(y_i)=[f(y_i ^-),f(y_i^+)]$.

Then, since $\widehat{u}_n$ is Lipschitz continuous, with a change of variable we obtain that
\begin{multline*}
\int_{I_i} \phi(y_i,\widehat{u}_{n}(x),\pi/2)\widehat{u}_n'(x)\,dx = \int_{\widehat{u}_n(\alpha_i)} ^{\widehat{u}_n(\beta_i)} \phi(y_i,s,\pi/2)\,ds
\\
= \int_{\widehat{u}_n(\alpha_i)} ^{f(y_i ^-)} \phi(y_i,s,\pi/2)\,ds
+\int_{\mathcal{J}_f(y_i)} \phi(y_i,s,\pi/2)\,ds
+\int_{f(y_i^+)} ^{\widehat{u}_n(\beta_i)} \phi(y_i,s,\pi/2)\,ds,
\end{multline*}
and therefore
\begin{eqnarray*}
\lefteqn{\hspace{-5em}
\limsup_{n\to +\infty} \left| \int_{I_i} \phi(y_i,\widehat{u}_{n}(x),\pi/2)\widehat{u}_n'(x)\,dx - \int_{\mathcal{J}_f(y_i)} \phi(y_i,s,\pi/2)\,ds\right| \leq}\\[0.5ex]
\qquad\qquad &\leq&\limsup_{n\to +\infty} \left|\int_{\widehat{u}_n(\alpha_i)} ^{f(y_i ^-)} \phi(y_i,s,\pi/2)\,ds\right| +\left|\int_{f(y_i^+)} ^{\widehat{u}_n(\beta_i)} \phi(y_i,s,\pi/2)\,ds\right|
\\[0.5ex]
&=&\left|\int_{f(\alpha_i )} ^{f(y_i ^-)} \phi(y_i,s,\pi/2)\,ds\right| +\left|\int_{f(y_i^+)} ^{f(\beta_i)} \phi(y_i,s,\pi/2)\,ds\right|
\\[0.5ex]
&\leq& M_0 \cdot |Df|(I_i \setminus \{y_i\}).
\end{eqnarray*}

As a consequence, from (\ref{est:L5J}) and the properties of $\Sigma_a$ we deduce that
$$\limsup_{n\to +\infty} |L_5^J|
\leq \Gamma_a ^+ \cdot |Df|((0,1)) + M_0 \cdot |Df|(\Sigma_a ^+ \setminus S_a ^+)\leq \Gamma_a ^+ \cdot |Df|((0,1)) + M_0 a.$$

Combining the previous estimates we conclude that
$$\limsup_{n\to +\infty} |L_1^J+L_2^J+L_3^J+L_4^J+L_5^J+L_6^J|\leq 2\Gamma_a ^+ \cdot |Df((0,1))| + M_0 (3a+\ep_a),$$
which proves (\ref{th:limsup-J+}). The proof of (\ref{th:limsup-J-}) is analogous.
\qed

\subsection{Proof of Theorem~\ref{thm:min}}

The proof of Theorem~\ref{thm:min} consists of two main parts. In the first part (estimate from below) we consider any sequence $\{u_{n}\}$ of functions such that $u_n\in PC_{1/n} (0,1)$ and we show that
\begin{equation}
\liminf_{n\to +\infty}\frac{\DPMF_{n}(\beta,f,(0,1),u_{n})}{\omega(n) ^{2}}\geq
\beta^{1/3} \int_{0}^{1}|f'(x)|^{2/3}\,dx.
\label{est:lim-DPMF-below}
\end{equation}

In the second part (estimate from above) we construct a sequence $\{u_{n}\}$ of functions such that $u_n\in PC_{1/n} (0,1)$ and
\begin{equation}
\limsup_{n\to +\infty}\frac{\DPMF_{n}(\beta,f,(0,1),u_{n})}{\omega(n)^{2}}\leq
\beta^{1/3}\int_{0}^{1}|f'(x)|^{2/3}\,dx.
\label{est:lim-DPMF-above}
\end{equation}

\subsubsection{Estimate from below}

\paragraph{\textmd{\textit{Interval subdivision and approximation of the forcing term}}}

Let us fix two real numbers $L>0$ and $\eta \in (0,1)$. For every integer $n\geq 2$ we set
\begin{equation}
L_{n}:=\frac{\lceil L n\omega(n)\rceil}{n\omega(n)}=\delta(n) \left\lceil \frac{L}{\delta(n)} \right\rceil,
\qquad
N_{n,L}:=\left\lfloor\frac{1}{L_n\omega(n)}\right\rfloor,
\qquad
A_{n,L}:= L_n \omega(n) N_{n,L}.
\label{def:LNA}
\end{equation}

We observe that $N_{n,L}$ is an integer, and that $L_{n}\to L^+$ and $A_{n,L}\to 1^-$ when $n\to +\infty$. We observe also that $[0,A_{n,L})$ is the disjoint union of the $N_{n,L}$ intervals of length $L_{n}\omega(n)$ defined by
\begin{equation}
I_{n,k}:=[(k-1)L_{n}\omega(n),kL_{n}\omega(n))
\qquad
\forall k\in\{1,\ldots,N_{n,L}\},
\label{defn:Ink}
\end{equation}
and we consider  the piecewise affine function $f_{n,L}:[0,A_{n,L})\to\R$ that interpolates the values of $f$ at the endpoints of these intervals, namely the function defined by
\begin{equation}
f_{n,L}(x):=M_{n,L,k}(x-(k-1)L_{n}\omega(n))+f((k-1)L_{n}\omega(n))
\qquad
\forall x\in I_{n,k},
\label{defn:fnL}
\end{equation}
where
\begin{equation}
M_{n,L,k}:=\frac{f(kL_{n}\omega(n))-f((k-1)L_{n}\omega(n))}{L_{n}\omega(n)}=\frac{1}{L_n \omega(n)} \int_{I_{n,k}} f'(y)\,dy.
\nonumber
\end{equation}

From the $H^{1}$ regularity of $f$ we deduce that the sequence $\{f_{n,L}\}$ converges to $f$ in the sense that
$$\lim_{n\to +\infty} \int_0 ^{A_{n,L}} (f_{n,L}'(x) - f'(x))^2 =0.$$
and
\begin{equation}
\lim_{n\to +\infty}\frac{1}{\omega(n)^{2}}\int_{0}^{A_{n,L}}(f_{n,L}(x)-f(x))^{2}\,dx=0.
\label{lim:fnLtof}
\end{equation}

In particular, we deduce that
\begin{equation}
\lim_{n\to +\infty} L_{n}\omega(n)\sum_{k=1}^{N_{n,L}}\phi(M_{n,L,k})=
\lim_{n\to +\infty} \int_{0}^{A_{n,L}}\phi(f_{n,L}'(x))\,dx=
\int_{0}^{1}\phi(f'(x))\,dx,
\label{lim:fnL'tof'}
\end{equation}
for every continuous function $\phi:\R\to \R$ that grows at most quadratically, namely such that
$$|\phi(s)|\leq C_1 + C_2 s^2 \qquad \forall s \in \R,$$
for some positive constants $C_1,C_2>0$.

Finally, from the inequality
\begin{equation}
(a+b)^{2}\geq(1-\eta)a^{2}+\left(1-\frac{1}{\eta}\right)b^{2}
\qquad
\forall\eta\in(0,1),
\quad
\forall (a,b)\in\re^{2},
\nonumber
\end{equation}
we obtain the estimate
\begin{equation}
\int_{0}^{1}(u_{n}-f)^{2}\,dx\geq
(1-\eta)\int_{0}^{A_{n,L}}(u_{n}-f_{n,L})^{2}\,dx
+\left(1-\frac{1}{\eta}\right)\int_{0}^{A_{n,L}}(f-f_{n,L})^{2}\,dx,
\nonumber
\end{equation}
from which we conclude that
\begin{eqnarray}
\DPMF_{n}(\beta,f,(0,1),u_{n})& \geq & (1-\eta)\DPMF_{n}(\beta,f_{n,L},(0,A_{n,L}),u_{n})
\nonumber
\\[0.5ex]
&  & 
\mbox{}+\left(1-\frac{1}{\eta}\right)\beta\int_{0}^{A_{n,L}}(f(x)-f_{n,L}(x))^{2}\,dx.
\label{est:DPMF-eta}
\end{eqnarray}

\paragraph{\textmd{\textit{Reduction to a common interval}}}

We prove that
\begin{equation}
\DPMF_{n}(\beta,f_{n,L},(0,A_{n,L}),u_{n})\geq
\omega(n)^{3}\sum_{k=1}^{N_{n,L}}\mu_{n}(\beta,L,M_{n,L,k}),
\label{ineq:DPMF-mLM}
\end{equation}
where $\mu_{n}$ is defined by (\ref{def:mu_n}). To this end, we begin by observing that
\begin{equation}
\DPMF_{n}(\beta,f_{n,L},(0,A_{n,L}),u_{n})\geq
\sum_{k=1}^{N_{n,L}}\DPMF_{n}(\beta,f_{n,L},I_{n,k},u_{n}),
\label{ineq:DPMF-sum}
\end{equation}
because the endpoints of the intervals $I_{n,k}$ are multiples of $1/n$, so passing from the left-hand side to the right-hand side we are just reducing the domain of integration, by neglecting the contribute of the discrete derivative in the last step of each interval $I_{n,k}$.

Each term in the sum can be reduced to the common interval $(0,L_{n})$ by introducing the function $v_{n,L,k}\in PC_{\delta(n)}(0,L_{n})$ defined by
\begin{equation}
v_{n,L,k}(y):=\frac{u_n((k-1)L_{n}\omega(n)+\omega(n) y)-f((k-1)L_{n}\omega(n))}{\omega(n)}
\qquad
\forall y\in(0,L_{n}).
\label{defn:vnLK}
\end{equation}

Indeed, with the change of variable $x=(k-1)L_{n}\omega(n)+\omega(n) y$, we obtain that
\begin{equation}
\int_{I_{n,k}}(u_n(x)-f_{n,L}(x))^{2}\,dx=
\omega(n)^{3}\int_{0}^{L_{n}}(v_{n,L,k}(y)-M_{n,L,k}\,y)^{2}\,dy
\nonumber
\end{equation}
and
\begin{equation}
\int_{(k-1)L_n\omega(n)} ^{kL_n\omega(n) - 1/n}
\log\left(1+D^{1/n}u_n(x)^{2}\right)dx
=\omega(n)^{3}\RDPM_{n}((0,L_{n}),v_{n,L,k}),
\nonumber
\end{equation}
and therefore we have that
\begin{eqnarray*}
\DPMF_{n}(\beta,f_{n,L},I_{n,k},u_{n}) & = &
\omega(n)^{3}\RDPMF_{n}(\beta,M_{n,L,k}\,x,(0,L_{n}),v_{n,L,k})
\\[1ex]
& \geq & \omega(n)^{3}\mu_{n}(\beta,L_{n},M_{n,L,k})
\\[1ex]
& \geq & \omega(n)^{3}\mu_{n}(\beta,L,M_{n,L,k}),
%\label{eqn:1000-620}
\end{eqnarray*}
where in the last inequality we exploited that $L_{n}\geq L$ for every $n$, and that $\mu_{n}$ is monotone with respect to the length of the interval. Plugging this inequality into (\ref{ineq:DPMF-sum}) we obtain (\ref{ineq:DPMF-mLM}).

\paragraph{\textmd{\textit{Convergence to minima of the limit problem}}}

We observe that for every positive real number $M_0>0$ there exists $n_0\in\N$ such that
\begin{eqnarray*}
\mu_{n}(\beta,L,M_{n,L,k})&\geq&\mu_{n}(\beta,L,\min\{|M_{n,L,k}|, M_0\})\nonumber\\
&\geq&\mu(4/3,\beta,L,\min\{|M_{n,L,k}|,M_0\})-\eta
\end{eqnarray*}
for every $n\geq n_0$ and every $k\in\{1,\ldots,N_{n,L}\}$, where $\mu$ is defined according to (\ref{def:mu}). Indeed, this estimate follows from Proposition~\ref{prop:mu_mu*}, and in particular from the monotonicity of $\mu_{n}$ with respect to $M$ and the uniform convergence in statement~(5).

Then, from the estimate from below in (\ref{eq:estimate_mu_below}) we obtain that
\begin{equation}\label{th:mLM-mu}
\mu_{n}(\beta,L,M_{n,L,k})\geq
\beta^{1/3}\min\{|M_{n,L,k}|,M_0\}^{2/3}L-\left(\frac{8}{3}\cdot 6^{2/3}+\eta\right).
\end{equation}
%and therefore in particular
%\begin{equation}
%c_{1}:=\frac{5}{4}\left(\frac{\alpha_{0}^{4}\beta}{3}\right)^{1/5}=
%10\left(\frac{2\beta}{27}\right)^{1/5}.
%\label{defn:c1}
%\end{equation}

\paragraph{\textmd{\textit{Conclusion}}}

Summing over $k$, from (\ref{ineq:DPMF-mLM}) and (\ref{th:mLM-mu}) we obtain that
\begin{multline*}
\frac{\DPMF_{n}(\beta,f_{n,L},(0,A_{n,L}),u_{n})}{\omega(n)^{2}} \;\geq\;  
\omega(n)\sum_{k=1}^{N_{n,L}}\mu_{n}(\beta,L,M_{n,L,k})
\\
 \geq\, \beta^{1/3} L \omega(n) \sum_{k=1}^{N_{n,L}}\min\{|M_{\ep,L,k}|,M_0\}^{2/3}
-\left(\frac{8}{3}\cdot 6^{2/3} +\eta\right) \omega(n) N_{n,L}.
\end{multline*}

Finally, plugging this estimate into (\ref{est:DPMF-eta}) we deduce that
\begin{eqnarray*}
\frac{\DPMF_{n}(\beta,f,(0,1),u_{n})}{\omega(n)^{2}} & \geq &
(1-\eta) \beta^{1/3} \frac{L}{L_{n}}\cdot L_{n}\omega(n)\sum_{k=1}^{N_{n,L}}|\min\{|M_{n,L,k}|,M_0\}|^{2/3}
\\[1ex]
& & \mbox{}-\omega(n) N_{n,L}\cdot(1-\eta)\left(\frac{8}{3}\cdot 6^{2/3}+\eta\right)
\\[1ex]
& & \mbox{}+\left(1-\frac{1}{\eta}\right)
\frac{\beta}{\omega(n) ^{2}}\int_{0}^{A_{n,L}}(f(x)-f_{n,L}(x))^{2}\,dx.
\end{eqnarray*}

Now we let $n\to +\infty$, and we exploit (\ref{lim:fnL'tof'}) in the first line, the fact that $\omega(n) N_{n,L}\to 1/L$ in the second line, and (\ref{lim:fnLtof}) in the third line. We conclude that
\begin{multline*}
\liminf_{n\to +\infty}\frac{\DPMF_{n}(\beta,f,(0,1),u_{n})}{\omega(n)^{2}} \geq\\
\geq (1-\eta)\left\{\beta^{1/3} \int_{0}^{1}\min\{|f'(x)|,M_0\}^{2/3}\,dx
-\frac{1}{L}\left(\frac{8}{3}\cdot 6^{2/3}+\eta\right)\right\}.
\end{multline*}

Finally, letting $\eta\to 0^{+}$, $L\to +\infty$ and $M_0\to +\infty$, we obtain exactly (\ref{est:lim-DPMF-below}).

\subsubsection{Estimate from above}

We show the existence of a sequence $\{u_{n}\}$ of functions such that $u_n\in PC_{1/n} (0,1)$ for which (\ref{est:lim-DPMF-above}) holds true. This amounts to proving the asymptotic optimality of all the steps in the proof of the estimate from below.

\paragraph{\textmd{\textit{Interval subdivision and approximation of the forcing term}}}

Let us fix again two real numbers $L>0$ and $\eta\in(0,1)$, and for every integer $n\geq 2$ let us define $L_{n}$ as in (\ref{def:LNA}), while, instead of $N_{n,L}$ and $A_{n,L}$, let us consider now
$$\widehat{N}_{n,L}:=\left\lceil \frac{1}{L_n \omega(n)}\right\rceil, \qquad \widehat{A}_{n,L}:=L_n\omega(n) \widehat{N}_{n,L}.$$

We observe that $\widehat{N}_{n,L}$ is an integer, and that $\widehat{A}_{n,L}\to 1^+$ when $n\to +\infty$.

Now we extend $f$ to $[0,\widehat{A}_{n,L})$ just by setting $f(x)=f(1)$ for $x \in [1,\widehat{A}_{n,L})$, and we consider the intervals $I_{n,k}$ as in (\ref{defn:Ink}), but now for $k\in\{1,\dots,\widehat{N}_{n,L}\}$, and the piecewise affine function $f_{n,L}:[0,\widehat{A}_{n,L})\to\R$ as in (\ref{defn:fnL}). Of course (\ref{lim:fnLtof}) and (\ref{lim:fnL'tof'}) are still valid if we replace $N_{n,L}$ and $A_{n,L}$ with $\widehat{N}_{n,L}$ and $\widehat{A}_{n,L}$.

Then we exploit the inequality
\begin{equation*}
(a+b)^{2}\leq(1+\eta)a^{2}+\left(1+\frac{1}{\eta}\right)b^{2}
\qquad
\forall\eta\in(0,1),
\quad
\forall (a,b)\in\re^{2},
\end{equation*}
and for every $u\in PC_{1/n} (0,\widehat{A}_{n,L})$ we obtain the estimate
\begin{eqnarray*}
\DPMF_{n}(\beta,f,(0,1),u) & \leq &
(1+\eta)\DPMF_{n}(\beta,f_{n,L},(0,\widehat{A}_{n,L}),u)
\\[1ex]
& & \mbox{}+\left(1+\frac{1}{\eta}\right)\beta\int_{0}^{\widehat{A}_{n,L}}(f(x)-f_{n,L}(x))^{2}\,dx.
\label{est:DPMF-eta-above}
\end{eqnarray*}

\paragraph{\textmd{\textit{Reduction to a common interval}}}

We claim that there exists $u_{n}\in PC_{1/n}(0,1)$ such that
\begin{equation*}
\DPMF_{n}(\beta,f_{n,L},(0,\widehat{A}_{n,L}),u_{n})  = 
\omega(n) ^{3}\sum_{k=1}^{\widehat{N}_{n,L}}\mu_{n}^{*}(\beta,L_{n},M_{n,L,k})
%\\
%& \leq & 
%\omega(n)^{3}\sum_{k=1}^{\widehat{N}_{\ep,L}}\mu_{n}^{*}(\beta,L,M_{n,L,k})+ \frac{2M^2+2}{\log n} + \frac{\beta M^2 \delta(n)^3}{3},
%\label{ineq:DPMF-mLM*}
\end{equation*}
where $\mu_{n}^{*}$ is defined by (\ref{def:mu_n*}).
%, and the inequality follows from (\ref{eq:monot_mu_n*}).

To this end, we observe that the equalities
\begin{eqnarray*}
\DPMF_{n}(\beta,f_{n,L},(0,\widehat{A}_{n,L}),u_{n}) & = &
\sum_{k=1}^{\widehat{N}_{\ep,L}}\DPMF_{n}(\beta,f_{n,L},I_{n,k},u_{n})
\\[1ex]
& = & 
\omega(n)^{3}\sum_{k=1}^{\widehat{N}_{n,L}}\RDPMF_{n}(\beta,M_{n,L,k}\,x,(0,L_{n}),v_{n,L,k})
\end{eqnarray*}
hold true for every $u_{n}\in PC_{1/n}(0,\widehat{A}_{n,L})$, provided that $u_{n}(x)$ and $v_{n,L,k}(x)$ are related by (\ref{defn:vnLK}) and
\begin{equation}\label{eq:un_BC}
D^{1/n}u_n(x)=0 \qquad \forall x \in [kL_n\omega(n) - 1/n, kL_n\omega(n))\quad \forall k \in \{1,\dots,\widehat{N}_{n,L}-1\}.
\end{equation}

Therefore it is enough to choose $u_{n}$ in such a way that $v_{n,L,k}$ coincides with a minimizer in the definition of $\mu_{n}^{*}(\beta,L_{n},M_{n,L,k})$ for every admissible choice of $k$.

Since $v_{n,L,k} \in PC_{\delta(n)} (0,L_n)$, it follows that the resulting function $u_n$ belongs to the space $PC_{1/n}(0,\widehat{A}_{n,L})$ and, due to the boundary conditions in (\ref{def:mu_n*}), we deduce that in all the nodes of the form $x=kL_{n}\omega(n)$ the function $u_n(x)$ is continuous and coincides with the forcing term $f(x)$. As a consequence, $u_n$ satisfies (\ref{eq:un_BC}). 
%As a consequence, the different pieces glue together in a continuous way, so the resulting function satisfies (\ref{eq:un_BC}). 

\paragraph{\textmd{\textit{Convergence to minima of the limit problem}}}

As in the case of the estimates from below we rely on Proposition~\ref{prop:mu_mu*}, and in particular on the monotonicity in statement~(2), the uniform convergence in statement~(5), and the estimates (\ref{est:mu_delta*}) and (\ref{eq:monot_mu_n*}), in order to deduce that for every $M_0>0$ there exists $n_{0}\in\N$ such that for every $n\geq n_0$ it holds
\begin{eqnarray*}
\mu_{n}^{*}(\beta,L_n,M_{n,L,k})%&\leq& \mu_{n}^{*}(\beta,L_n,M_{n,L,k})  \\
&\leq& \mu_n ^{*}(\beta,L,M_{n,L,k})+ \frac{\log(2M_{n,L,k} ^2+2)}{\log n} + \frac{\beta M_{n,L,k} ^2 \delta(n)^3}{3}\\
&\leq& \mu ^{*}(4/3,\beta,L,M_{n,L,k})+ \frac{\log(2M_0 ^2+2)}{\log n} + \frac{\beta M_0 ^2 \delta(n)^3}{3} + \eta
\end{eqnarray*}
for every $k\in\{1,\ldots,\widehat{N}_{n,L}\}$ such that $|M_{n,L,k}|\leq M_0$, and
\begin{eqnarray*}
\mu_{n}^{*}(\beta,L_n,M_{n,L,k})&\leq& \left(\frac{M_{n,L,k}}{M_0}\right)^2 \mu_{n}^{*}(\beta,L_n,M_0)\\
&\leq&\left(\frac{M_{n,L,k}}{M_0}\right)^2 \left(\mu_n ^{*}(\beta,L,M_0) + \frac{\log(2M_0 ^2+2)}{\log n} + \frac{\beta M_0 ^2 \delta(n)^3}{3} \right)\\
&\leq&\left(\frac{M_{n,L,k}}{M_0}\right)^2 \left(\mu^{*}(4/3,\beta,L,M_0) + \frac{\log(2M_0 ^2+2)}{\log n} + \frac{\beta M_0 ^2 \delta(n)^3}{3}+ \eta\right)
\end{eqnarray*}
for every $k\in\{1,\ldots,\widehat{N}_{n,L}\}$ such that $|M_{n,L,k}|> M_0$.

Now we exploit the estimate from above in (\ref{eq:estimate_mu_above}), and we obtain that
\begin{equation*}
\mu^{*}(4/3,\beta,L,M_{n,L,k})\leq
\beta^{1/3}|M_{n,L,k}|^{2/3}L+2,
\end{equation*}
and
\begin{equation*}
\left(\frac{M_{n,L,k}}{M_0}\right)^2 \mu^{*}(4/3,\beta,L,M_0)\leq
\beta^{1/3}\frac{(M_{n,L,k})^2}{M_0^{4/3}}L+2\left(\frac{M_{n,L,k}}{M_0}\right)^2 ,
\end{equation*}

We can unify the previous estimates in the following inequality
\begin{eqnarray*}
\mu_{n}^{*}(\beta,L,M_{n,L,k})&\leq& \beta^{1/3} \frac{(M_{n,L,k})^2}{\min\{|M_{n,L,k}|,M_0\}^{4/3}} L \\
&&\mbox{} +\left(\frac{\max\{|M_{n,L,k}|,M_0\}}{M_0}\right)^2 \left( 2+\eta+\frac{\log(2M_0 ^2+2)}{\log n} + \frac{\beta M_0 ^2 \delta(n)^3}{3}\right),
\end{eqnarray*}
that holds for every $n\geq n_0$ and every $k\in\{1,\ldots,\widehat{N}_{n,L}\}$.

\paragraph{\textmd{\textit{Conclusion}}}

As in the estimate from below, we sum over $k$ and we conclude that
\begin{eqnarray*}
\frac{\DPMF_{n}(\beta,f,(0,1),u_{n})}{\omega(n)^{2}} & \leq &
(1+\eta)\beta^{1/3}\frac{L}{L_n} L_n \omega(n) \sum_{k=1}^{\widehat{N}_{n,L}}\frac{(M_{n,L,k})^2}{\min\{|M_{n,L,k}|,M_0\}^{4/3}}
\\[1ex]
& & \mbox{}+\frac{1+\eta}{L_n}\left(2+\eta+\frac{\log(2M_0 ^2+2)}{\log n} + \frac{\beta M_0 ^2 \delta(n)^3}{3}\right)
\\[0.5ex]
& & \quad\mbox{}\cdot\omega(n) L_n \sum_{k=1}^{\widehat{N}_{n,L}}\left(\frac{\max\{|M_{n,L,k}|,M_0\}}{M_0}\right)^2
\\[1ex]
& & \mbox{}+\left(1+\frac{1}{\eta}\right)
\frac{\beta}{\omega(n)^{2}}\int_{0}^{\widehat{A}_{n,L}}(f(x)-f_{n,L}(x))^{2}\,dx.
\end{eqnarray*}

Letting $n\to +\infty$, from (\ref{lim:fnL'tof'}) and (\ref{lim:fnLtof}) (with $\widehat{N}_{n,L}$ and $\widehat{A}_{n,L}$ in place of $N_{n,L}$ and $A_{n,L}$) we deduce that this family $\{u_{n}\}$ satisfies
\begin{eqnarray*}
\limsup_{n\to +\infty}\frac{\DPMF_{n}(\beta,f,(0,1),u_{n})}{\omega(n) ^{2}} &\leq& 
(1+\eta)\beta^{1/3}\int_{0}^{1}\frac{|f'(x)|^{2}}{\min\{|f'(x)|,M_0\}^{4/3}}\,dx\\[0.5ex]
& & \mbox{}+\frac{(1+\eta)(2+\eta)}{L} \int_{0} ^{1} \left(\frac{\max\{|f'(x)|,M_0\}}{M_0}\right)^2 \, dx.
\end{eqnarray*}

Now we observe that the right-hand side tends to the right-hand side of (\ref{est:lim-DPMF-above}) when $\eta\to 0^{+}$, $L\to +\infty$ and $M_0 \to +\infty$. Therefore, with a standard diagonal procedure we can find a family $\{u_{n}\}$ for which exactly (\ref{est:lim-DPMF-above}) holds true.
\qed

\subsection{Proof of Theorem \ref{thm:blow-up}}

\subsubsection{Proof of statement (1)}

The proof relies on Proposition~\ref{prop:convergence_minimizers}, applied with
$$M=f'(x_\infty),\qquad
A_k=-\frac{x_{n_k}}{\omega(n_k)}, \qquad
B_k=\frac{1-x_{n_k}}{\omega(n_k)},\qquad w_k=w_{n_k}$$
and
$$g_k(x)=\frac{f(x_{n_k}+\omega(n_k) x)-f(x_{n_k})}{\omega(n_k)}$$

Let us check that all the assumptions are verified. First of all, we observe that $A_k\to -\infty$ and $B_k\to +\infty$, while $g_k(x)\to Mx$ uniformly on bounded sets because $f \in C^1([0,1])$.

We also have that $w_{k} \in PC_{\delta(n_k)}(A_k,B_k)$, because $u_{n_k} \in PC_{1/n_k}(0,1)$ and $n_k x_{n_k}\in\Z$. Moreover, with a change of variable, we obtain that
\begin{equation}\label{eq:change_variable_wn}
\DPMF_{n_k}(\beta,f,(0,1),u_{n_k})=\omega(n_k) ^3 \RDPMF_{n_k}(\beta,g_k,(A_k,B_k),w_k),\end{equation}
hence $w_k$ is a minimizer for $\RDPMF_{n_k}(\beta,g_k,(A_k,B_k),\cdot)$.

Finally, estimate (\ref{hp:RDPMF<C/delta}) follows from (\ref{eq:change_variable_wn}) and Theorem~\ref{thm:min}.

Therefore from Proposition~\ref{prop:convergence_minimizers} we deduce that there exists a subsequence $\{w_{n_{k_h}}\}$ that converges locally strictly in $BV_{loc}(\R)$ to an entire local minimizer $w_\infty$ for $\JF(4/3,\beta,Mx,\cdot,\cdot)$.

By Proposition~\ref{prop:local_minimizers}, $w_\infty$ is an oblique translation of the canonical $(H,V)$-staircase, as required.

\subsubsection{Proof of statement (2)}

We observe that $v_n(x)=w_n(x)-w_n(0)$, so the behavior of the sequence $\{v_n\}$ can be deduced from that of $\{w_n\}$.

In particular, if $w_{n_k}\auto w_\infty$ and $w_\infty=S_{H,V}(x-H\tau_0)+V\tau_0$, for some $\tau_0\in (-1,1)$, then $w_\infty$ is continuous at $0$, so from the strict convergence we deduce that $v_{n_k}(x)\auto w_\infty(x)-w_\infty(0)=S_{H,V}(x-H\tau_0)$, which is a graph translation of horizontal type of $S_{H,V}$, as required.

On the other hand, if $w_\infty = S_{H,V}(x-H)+V$ is the graph translation corresponding to $\tau_0=\pm 1$, then from the strict convergence we only deduce that
$$-|V|\leq \liminf_{k\to +\infty} w_{n_k}(0)\leq \limsup_{k\to +\infty} w_{n_k}(0)\leq |V|.$$

However, from the last part of Proposition~\ref{prop:convergence_minimizers} we deduce that the only possible limit points for $\{w_{n_k}(0)\}$ are $\pm V$. Hence, up to the extraction of another subsequence, we have that either $v_{n_k}\auto S_{H,V}(x-H)$ or $v_{n_k}\auto S_{H,V}(x-H)+2V=S_{H,V}(x+H)$, and both these functions are graph translation of horizontal type of $S_{H,V}$, as required.

\subsubsection{Proof of statement (3) and (4)}

The last two statement can be proved exactly as in \cite{GP:fastPM-CdV}, so we just recall the main steps. First of all, one can obtain (\ref{fakebu_to_w}) with $w$ equal to the canonical staircase $S_{H,V}$ by choosing the point $x_n'\in [x_n - H\omega(n),x_n+H\omega(n)]$ that minimize some distance between $S_{H,V}$ and the function
$$y\mapsto \frac{u_n(x_n'+\omega_n y)-f(x_n')}{\omega(n)}.$$

At this point one can obtain any other oblique translation of $S_{H,V}$ with a suitable translation of the points $x_n'$, thus completing the proof of statement (3).

Statement (4) can be proved using the same sequences $\{x_n'\}$ coming from statement (3) and recalling the equality $v_n(x)=w_n(x)-w_n(0)$ that we already exploited in the proof of statement (2).

\subsubsection*{\centering Acknowledgments}

The author is a member of the \selectlanguage{italian} ``Gruppo Nazionale per l'Analisi Matematica, la Probabilità e le loro Applicazioni'' (GNAMPA) of the ``Istituto Nazionale di Alta Matematica'' (INdAM). 

\selectlanguage{english}

\label{NumeroPagine}

\end{document}